\documentclass[11pt,a4paper,leqno]{amsart}
\usepackage{amssymb,amsmath,amsthm}

\usepackage[mathscr]{eucal} 
\usepackage{graphicx}
\usepackage[colorlinks=true,linkcolor=Red,citecolor=Green,urlcolor=Fuchsia]{hyperref}
\usepackage{hyperref}
\usepackage{color}
\usepackage{pdfsync}
\usepackage{enumerate} 
\usepackage{dsfont}
\usepackage{esint}
\usepackage[utf8]{inputenc}
\usepackage{csquotes}
\MakeOuterQuote{"}
\usepackage[dvipsnames]{xcolor}

\allowdisplaybreaks
\theoremstyle{plain}
\newtheorem{theorem}{Theorem}[section]
\newtheorem{definition}[theorem]{Definition}

\newtheorem{lemma}[theorem]{Lemma}
\newtheorem{corollary}[theorem]{Corollary}
\newtheorem{proposition}[theorem]{Proposition}
\theoremstyle{remark}
\newtheorem{remark}[theorem]{Remark}

\newtheorem{claim}[theorem]{Claim}

\numberwithin{equation}{section}

\def\dd{\,\mathrm  d}
\newcommand{\loc}{\rm loc}
\renewcommand{\div}{\operatorname{div}}
\newcommand{\curl}{\operatorname{curl}}

\newcommand{\supp}{\operatorname{supp}}

\renewcommand{\leq}{\leqslant}
\renewcommand{\geq}{\geqslant}
\renewcommand{\epsilon}{\varepsilon}
\renewcommand{\phi}{\varphi}
\newcommand{\norm}[1]{\left\lVert#1\right\rVert}

\def\Tend#1#2{\mathop{\longrightarrow}\limits_{#1\rightarrow#2}}

\date\today

\author[M. Ménard]{Matthieu Ménard}

\email{matthieu.menard@univ-grenoble-alpes.fr}

\address{Univ. Grenoble Alpes, CNRS, Institut Fourier, F-38000 Grenoble, France.}

\title[A monokinetic spray model with gyroscopic effects.]{Mean-field limit derivation of a monokinetic spray model with gyroscopic effects.}

\begin{document}
\begin{abstract}
In this paper we derive a two dimensional spray model with gyroscopic effects as the mean-field limit of a system modeling the interaction between an incompressible fluid and a finite number of solid particles. This spray model has been studied by Moussa and Sueur (Asymptotic Anal., 2013), in particular the mean-field limit was established in the case of $W^{1,\infty}$ interactions. \\
First we prove the local in time existence and uniqueness of strong solutions of a monokinetic version of the model with a fixed point method. Then we adapt the proof of Duerinckx and Serfaty (Duke Math. J., 2020) to establish the mean-field limit to the spray model in the monokinetic regime in the case of Coulomb interactions.
\end{abstract}

\maketitle

\setlength{\oddsidemargin}{0pt} 
\setlength{\evensidemargin}{0pt}

\section{Introduction}

The purpose of this paper is to establish the mean-field limit derivation of a system of partial differential equations introduced by Moussa and Sueur in \cite{MoussaSueur} to describe a two dimensional spray modeled by an incompressible fluid and a dispersed phase of solid particles with the following interactions: The fluid particles move through the velocity field $V$ generated by the fluid and the solid particles whereas the solid particles are submitted to a gyroscopic effect related to their velocities and to $V$. We define
\begin{equation}\label{definition_g}
g(x) := -\frac{1}{2\pi}\ln|x|
\end{equation}
as the opposite of the Green kernel on the plane. Let $\omega(t,x)$ be the vorticity of the fluid and $f(t,x,\xi)$ be the density of solid particles, then this system  can be written
\begin{equation}\label{edp_ordre_2}
\left\{
\begin{aligned}
& \partial_t \omega + \div(\omega V) = 0 \\
& \partial_t f + \xi \cdot \nabla_x f + \div_\xi\left((\xi - V)^\bot f\right) = 0 \\
& V = -\nabla^\bot g\ast(\omega + \rho) \\
& \rho(t,x) = \int_{\mathbb{R}^2} f(t,x,\xi) \dd\xi
\end{aligned}\right.
\end{equation}
where $u^\bot := (-u_2,u_1)$ and $\rho$ is the space density of solid particles.

Replacing $\nabla^\bot g$ with a $W^{1,\infty}$ kernel, Moussa and Sueur derived these equations as the mean-field limit of a model describing the dynamics of a finite number of particles moving in an incompressible fluid (see \cite[Corollary 1]{MoussaSueur}). Namely, for $N$ solid particles immersed in a fluid of vorticity $\omega_N(t,x)$ with initial condition $\omega_0$, if the number of particles becomes large and if at time zero their empirical measure $f_N(0)$ is close to a regular density $f_0$, then for any time $t$ $(f_N(t),\omega_N(t))$ is also close to the solution $(f(t),\omega(t))$ of \eqref{edp_ordre_2} starting from $(f_0,\omega_0)$. The system modeling the interaction between a fluid of vorticity $\omega_N$ and $N$ solid particles with positions $q_1,...,q_N$ and velocities $p_1,...,p_N$ is the following:
\begin{equation}
\label{edo}
\left\{
\begin{aligned}
& \partial_t \omega_N + \div(\omega_N V_N) = 0 \\
& \dot q_i = p_i \\
&\dot p_i = p_i^\bot - \nabla g \ast \omega_N(q_i) - \frac{1}{N}\underset{j\neq i}{\sum_{j=1}^N} \nabla g (q_i(t)-q_j(t)) \\
& V_N = -\nabla^\bot g \ast (\omega_N + \rho_N) \\
& \rho_N = \frac{1}{N}\sum_{k=1}^N\delta_{q_i} \\
\end{aligned}\right.
\end{equation} 
where $\displaystyle{\frac{1}{N}}$ represents both the mass of a solid particle and the circulation of velocity around it. This model was established by Glass, Lacave and Sueur in \cite{GlassLacaveSueur} by looking at a rigid body in a fluid and assuming that its size is going to zero. Its well-posedness was studied in \cite{LacaveMiot} by Lacave and Miot. Remark that we can formally obtain this system from System \eqref{edp_ordre_2} if we take 
\begin{equation*}
f = \frac{1}{N}\sum_{i=1}^N \delta_{q_i}\otimes\delta_{p_i}.
\end{equation*}
By Theorem 1.2 of \cite{LacaveMiot}, we know that there exists a unique global weak solution of System \eqref{edo} on $\mathbb{R}_+\times \mathbb{R}^2$.

In this paper we adapt the proof of Duerinckx and Serfaty in \cite{Serfaty} to extend the mean-field convergence result of \cite{MoussaSueur} for the true Coulombian interaction, that is we prove the convergence of \eqref{edo} to \eqref{edp_ordre_2} in the monokinetic regime, or more precisely to the following system:
\begin{equation}
\label{edp}
\left\{
\begin{aligned}
& \partial_t \omega + \div(\omega V) = 0 \\
& \partial_t \rho + \div(\rho v) = 0 \\
& \partial_t v + (v\cdot\nabla) v = (v - V)^\bot \\
& V = -\nabla^\bot g\ast(\omega + \rho).
\end{aligned}\right.
\end{equation}

It can be obtained by taking formally $f(t,x,\xi) = \rho(t,x)\otimes \delta_{\xi=v(t,x)}$ in System \eqref{edp_ordre_2}. A rigorous derivation of System \eqref{edp} from \eqref{edp_ordre_2} was proved in \cite{MoussaSueur} replacing $\nabla^\bot g$ with a $W^{1,\infty}$ kernel.

Before establishing the mean-field limit, we will justify the local in time existence and uniqueness of strong solutions of System \eqref{edp}. The local well-posedness of Euler-Poisson system (that is the system we get if we take $\omega = 0$ and add a pressure term in the equation on $v$) was studied in \cite{Makino} in the case $d=3$ using the usual estimates on hyperbolic systems that were proved in \cite{Kato}. In Section \ref{section:2} we extend this result to System \eqref{edp}. We will not study the existence of weak solutions of our system, for more details on this subject one can refer to the bibliography of the appendix of \cite{Serfaty}. 

Mean-field limits for regular kernels were first established by compactness arguments in \cite{BraunHepp,NeuzertWick} or by optimal transport theory and Wasserstein distances by Dobrushin in \cite{Dobrushin}. The latest method is the one used in \cite{MoussaSueur} to prove the mean-field convergence of \eqref{edo} to \eqref{edp_ordre_2}. In the Coulomb case, the kernel is no longer regular and their proof no longer holds. However, there are other works which prove mean-field limits for some systems with Coulombian or Riesz interactions. Let us consider $N$ particles $x_1,...,x_N$ satisfying the differential equations:
\begin{equation}
\label{edo_generique_ordre_1}
\dot x_i = \frac{1}{N}\underset{j\neq i}{\sum_{j=1}^N}K(x_i - x_j).
\end{equation}
Then the mean-field limit to a density $\mu$ satisfying
\begin{equation}
\label{edp_generique_ordre_1}
\partial_t \mu + \div((K\ast\mu)\mu) = 0
\end{equation}
has been rigorously justified in different cases:

Schochet proved in \cite{Schochet1} the mean-field limit of the point vortex system (that is $\displaystyle{K = \frac{1}{2\pi}\frac{x^\bot}{|x|^2}}$ in dimension two) to a measure-valued solution of the Euler equations up to a subsequence, using arguments previously developed in \cite{Schochet2} and \cite{Delort} to prove existence of such solutions. This result was extended later to include convergence to vortex sheets in \cite{LiuXin}.

For sub-coulombic interactions, that is $|K(x)|, |x||\nabla K(x)| \leq C|x|^{-\alpha}$ with $0 < \alpha < d-1$, the mean-field limit was proved by Hauray in \cite{Hauray} assuming $\div(K) = 0$ and using a Dobruschin-type approach. It was also used by Carillo, Choi and Hauray to deal with the mean-field limit of some aggregation models in \cite{CarrilloChoiHauray,CarrilloChoiHauraySalem}.

In \cite{Duerinckx} Duerinckx gave another proof of the mean-field limit of several Riesz interaction gradient flows using a "modulated energy" that was introduced by Serfaty in \cite{Serfaty2017}. Together they also established the mean-field limit of Ginzburg-Landau vortices with pinning and forcing effects in \cite{DuerinckxSerfaty}.

In \cite{Serfaty}, Serfaty proved the mean-field convergence of such systems where $K$ was a kernel given by Coulomb, logarithmic or Riesz interaction, that is $K = \nabla g$ for $g(x) = |x|^{-s}$ with $\max(d-2,0) \leq s < d$ for $d \geq 1$ or $g(x) = -\ln|x|$ for $d = 1$ or $2$. For this purpose $K\ast\mu$ is supposed to be Lipschitz.

Rosenzweig proved in \cite{Rosenzweig} the mean-field convergence of the point vortex system without assuming Lipschitz regularity of the limit velocity field, using the same energy as in \cite{Serfaty} with refined estimates. Remark that it ensures that the point vortex system converges to any Yudovich solutions of the Euler equations (see \cite{Yudovich}). This result was extended later for higher dimensional systems ($d \geq 3$) in \cite{Rosenzweig4} by the same author.

Numerous mean-field limit results were proved for interacting particles with noise with regular or singular interaction kernels in \cite{BermanOnheim,BolleyChafaiJoaquin,BreschJabinWang,BreschJabinWang2,CarilloFerreiraPrecioso,FournierHaurayMischler,LiLiuYu,Osada}.

For systems of order two satisfying Newton's second law:
\begin{equation}
\label{edo_generique_ordre_2}
\ddot x_i = \frac{1}{N}\underset{j \neq i}{\sum_{j=1}^N} K(x_i - x_j)
\end{equation}
the mean-field convergence to Vlasov-like equations remains open in the Coulombian case but was established for some singular kernels:

In \cite{HaurayJabin1,HaurayJabin2}, Hauray and Jabin treated the case of some sub-coulombian interactions, or more precisely they considered a kernel $K = \nabla g$ where $|\nabla g(x)| \leq C|x|^{-s}$ and $|\nabla g(x)| \leq C|x|^{-s-1}$ where $0 < s < 1$. For this purpose they used the same kind of arguments Hauray used in \cite{Hauray}.

In \cite{JabinWang1,JabinWang2}, Jabin and Wang treated the case of bounded and $W^{-1,\infty}$ gradients.

In \cite{BoersPickl,Lazarovici,LazaroviciPickl,HuangLiuPickl} the same kind of results is proved with some cutoff of the interaction kernel.

In the appendix of \cite{Serfaty}, Duerinckx and Serfaty treated the case of particles with Coulombian interactions converging to the Vlasov equations in the monokinetic regime, that is the pressureless Euler-Poisson equations. This was used later by Carillo and Choi in \cite{CarilloChoi} to prove the mean-field limit of some swarming models with alignment interactions.

In \cite{HanKwanIacobelli}, Han-Kwan and Iacobelli proved the mean-field limit of particles satisfying Newton's second law to the Euler equations in a quasineutral regime or in a gyrokinetic limit. This result was improved later by Rosenzweig in \cite{Rosenzweig2} who treated the case of quasineutral regime for a larger choice of scaling between the number of particles and the coupling constant.

For a general introduction to the subject of mean-field limits one can have a look at the reviews \cite{Golse,Jabin}.

\subsection{Main results}

If $\nu$ is a probability measure on $\mathbb{R}^2$, we will denote
\begin{equation*}
\nu^{\otimes 2} := \nu \otimes \nu.
\end{equation*}
Recall that $g$ is the opposite of the Green kernel on the plane:
\begin{equation*}
g(x) := -\frac{1}{2\pi}\ln|x|.
\end{equation*}
$\Delta$ will denote the diagonal of $(\mathbb{R}^2)^2$:
\begin{equation*}
\Delta := \{(x,x) \; ; \; x \in \mathbb{R}^2\}.
\end{equation*}
The main result in this paper is Theorem \ref{theoreme_champ_moyen} which proves the mean-field limit of solutions of System \eqref{edo} to solutions of \eqref{edp} with some regularity assumptions. We will use the following definition of weak solutions:
\begin{definition}
\label{def_sol_faible_edp}
We say that $(\rho,\omega,v)$ is a weak solution of \eqref{edp} if
\begin{enumerate}
\item $\rho,\omega \in \mathcal{C}([0,T],L^1\cap L^\infty(\mathbb{R}^2,\mathbb{R}))$ with compact supports.
\item For all $t \in [0,T]$, $\displaystyle{\int_{\mathbb{R}^2}\rho(t) = \int_{\mathbb{R}^2} \omega(t) = 1}$.   
\item $v \in W^{1,\infty}([0,T]\times\mathbb{R}^2,\mathbb{R}^2)$
\item The equation on the velocity is satisfied almost everywhere and the continuity equations are satisfied in the sense of distributions, that is for every $\phi \in W^{1,\infty}([0,T],\mathcal{C}^1_c(\mathbb{R}^2))$ and for every $t \in [0,T]$, we have:
\begin{equation}\label{def_solution_faible_eq_continuite}
\begin{aligned}
\int_{\mathbb{R}^2}(\rho(t)\phi(t) - \rho_0\phi(0))&= \int_0^t\int_{\mathbb{R}^2} \rho(s,x)(\partial_s \phi + \nabla \phi \cdot v)(s,x)\dd x \dd s \\
\int_{\mathbb{R}^2}(\omega(t)\phi(t) - \omega_0\phi(0)) &= \int_0^t\int_{\mathbb{R}^2} \omega(s,x)(\partial_s \phi + \nabla \phi \cdot V)(s,x)\dd x \dd s
\end{aligned}
\end{equation}
\end{enumerate}
\end{definition}
Remark that by conservation of mass it is enough to ask 
\begin{equation*}
\int_{\mathbb{R}^2}\rho_0 = \int_{\mathbb{R}^2} 
\omega_0 = 1
\end{equation*}
to get Assumption $(2)$.

In Section \ref{section:2} we will prove existence and uniqueness of solutions of \eqref{edp} in a space strictly included in $\mathcal{C}([0,T],L^1\cap L^\infty)^2\times W^{1,\infty}$ (see Theorem \ref{main_theorem_well_posedness}). For the microscopic system \eqref{edo}, we will use the following definition of weak solutions, introduced in \cite{LacaveMiot}:
\begin{definition}
\label{def_solution_faible_edo}
$(\omega_N,Q_N,P_N)$ is a weak solution of \eqref{edo} on $[0,T]$ if 
\begin{enumerate}
\item $\omega_N \in L^\infty([0,T],L^1\cap L^\infty)\cap \mathcal{C}([0,T],L^\infty-w^*)$ with compact support.
\item For all $t \in [0,T]$, $\displaystyle{\int_{\mathbb{R}^2}\omega_N(t) = 1}$.
\item $q_1,...,q_N \in \mathcal{C}^2([0,T],\mathbb{R}^2)$
\item The partial differential equation on $\omega_N$ is satisfied in the sense of distributions (which means that it also verifies \eqref{def_solution_faible_eq_continuite}) and the ordinary differential equations are satisfied in the classical sense.
\end{enumerate}
\end{definition}
Remark that by conservation of mass it is enough to ask 
\begin{equation*}
\int_{\mathbb{R}^2} 
\omega_{N}(0) = 1
\end{equation*}
to get Assumption $(2)$.

\begin{remark}\label{solution_edo_bien_definie}
By Theorems 1.4 and 1.5 of \cite{LacaveMiot} we know that for $\omega_N(0) \in L^\infty(\mathbb{R}^2)$ compactly supported and $q_1(0),...,q_N(0)$ distinct outside of the support of $\omega_N(0)$ there exists a unique weak solution of \eqref{edo} on $[0,T]$ for any $T > 0$ and no collision between the solid particles occurs in finite time. It follows by \cite[Corollary A.2]{LacaveMiot} that for all $1 \leq p \leq \infty$, $\norm{\omega_N(t)}_{L^p} = \norm{\omega_N(0)}_{L^p}$.
\end{remark}

\begin{remark}
One could replace the compact support assumptions by some logarithmic decrease of the solutions $\omega$ and $\rho$ at infinity as done in \cite{Duerinckx} and \cite{Rosenzweig} but for the sake of simplicity we will only consider solutions with compact support.
\end{remark}

In order to show that the limit of a sequence $(\omega_N,Q_N,P_N)$ of solutions of \eqref{edo} converges to a solution $(\omega,\rho,v)$ of \eqref{edp}, we will control a modulated energy similar to the one defined in \cite{Serfaty}. Let $X_N = (x_1,...,x_N) \in (\mathbb{R}^2)^N$ be such that $x_i \neq x_j$ if $i \neq j$ and let $\mu$ be a $L^1 \cap L^\infty$ probability density with compact support, then the following quantity is well defined:
\begin{equation}
\label{definition_energie_modulee_F}
\mathcal{F}(X_N,\mu) := \int_{\mathbb{R}^2\times\mathbb{R}^2\backslash \Delta} g(x-y)\left(\mu-\sum_{i=1}^N\delta_{x_i}\right)(\dd x)\left(\mu-\sum_{i=1}^N\delta_{x_i}\right)(\dd y).
\end{equation}
This is the "modulated energy"  used in \cite{Serfaty,Rosenzweig} to prove the mean-field limit of \eqref{edo_generique_ordre_1} to \eqref{edp_generique_ordre_1}. As we will see later this quantity controls the distance between $\mu$ and the empirical distribution on $X_N$ in a weak sense. More precisely we have the following proposition proved in \cite{Serfaty} (number 3.6 in the article):
\begin{proposition}[proved in \cite{Serfaty}]\label{coercivite_F}
For any $0 < \theta \leq 1$, there exists $\lambda > 0$ and $C > 0$ such that for $\xi$ smooth and $\mu \in L^\infty$ probability density with compact support,
\begin{multline*}
\left|\int_{\mathbb{R}^2}\xi \left(\mu - \frac{1}{N}\sum_{i=1}^N\delta_{x_i}\right)\right| \leq C\bigg(|\xi|_{\mathcal{C}^{0,\theta}}N^{-\lambda} \\ + \norm{\nabla \xi}_{L^2}\bigg(\mathcal{F}(X_N,\mu) + (1+\norm{\mu}_{L^\infty})N^{-1}+\frac{\ln(N)}{N}\bigg)^\frac{1}{2}\bigg)
\end{multline*}
where 
\begin{align*}
|\xi|_{\mathcal{C}^{0,\theta}} := \underset{x \neq y}{\sup} \frac{|\xi(x)-\xi(y)|}{|x-y|^\theta}.
\end{align*}
\end{proposition}

\begin{remark}
In \cite[Proposition 3.6]{Serfaty} the coercivity inequality is stated with the Hölder norm $\norm{\xi}_{\mathcal{C}^{0,\theta}}$  but by inequality \cite[Inequality (3.27)]{Serfaty} we can replace this Hölder norm by the semi-norm $|\xi|_{\mathcal{C}^{0,\theta}}$.
\end{remark}

We will also need the following functional inequality, proved by Serfaty in \cite{Serfaty} (number 1.1 in the article). 
\begin{proposition}\label{main_functional_inequality_serfaty}
There exists $\lambda, C > 0$ such that for any probability density  $\mu \in L^\infty$ with compact support, $\psi \in W^{1,\infty}(\mathbb{R}^2,\mathbb{R}^2)$ and $X_N \in (\mathbb{R}^2)^N$, we have
\begin{multline*}
\int_{\mathbb{R}^2\times\mathbb{R}^2\backslash \Delta} (\psi(x)-\psi(y))\cdot\nabla g(x-y)\dd\bigg(\frac{1}{N}\sum_{i=1}^N \delta_{x_i} -\mu\bigg)(x)\dd\bigg(\frac{1}{N}\sum_{i=1}^N \delta_{x_i} -\mu\bigg)(y) \\
\leq C\norm{\psi}_{W^{1,\infty}}(\mathcal{F}(X_N,\mu) + (1+\norm{\mu}_{L^\infty})N^{-\lambda}).
\end{multline*}
\end{proposition}
This proposition is one of the main result of \cite{Serfaty} as it is used to perform a Grönwall estimate on the modulated energy from which the mean-field result is deduced.

Now let $\rho,\omega,\omega_N$ be $(L^1 \cap L^\infty)(\mathbb{R}^2,\mathbb{R})$ probability densities with compact supports, $v \in W^{1,\infty}(\mathbb{R}^2,\mathbb{R}^2)$, $Q_N, P_N \in (\mathbb{R}^2)^N$ be such that $q_i \neq q_j$ if $i \neq j$. We define:
\begin{equation*}
\begin{aligned}
\mathcal{H}(\omega,\rho,v,\omega_N,Q_N,P_N)& \\
:=& \frac{1}{N}\sum_{i=1}^N |v(q_i) - p_i|^2  \\
&+ \iint_{(\mathbb{R}^2\times\mathbb{R}^2)\backslash \Delta} g(x-y)(\rho + \omega - \rho_N - \omega_N)^{\otimes 2}(\dd x\dd y) \\
&+ \iint_{\mathbb{R}^2\times\mathbb{R}^2} g(x-y)(\omega-\omega_N)(x)(\omega-\omega_N)(y)\dd x \dd y \\
&+ \norm{\omega - \omega_N}^2_{L^2} + B N^{-\gamma}
\end{aligned}
\end{equation*}
where 
\begin{equation*}
\rho_N := \frac{1}{N}\sum_{i=1}^N \delta_{q_i}
\end{equation*}
and $\gamma$ and $B$ are constants ensuring that $\mathcal{H}$ is nonnegative, as explained in the following result:
\begin{proposition}\label{asmptotic_positivity_energy}
For any $0 < \gamma < 1$, there exists a constant $B$ depending only on  $\gamma$, $\norm{\omega}_{L^1 \cap L^\infty}$, $\norm{\rho}_{L^1 \cap L^\infty}$ and $\underset{N}{\sup} \norm{\omega_N}_{L^1 \cap L^\infty}$ such that:
\begin{equation}\label{lower_bound_potential_energy}
\iint_{(\mathbb{R}^2\times\mathbb{R}^2)\backslash \Delta} g(x-y)(\rho + \omega - \rho_N - \omega_N)^{\otimes 2}(\dd x\dd y) + B N^{-\gamma} \geq 0.
\end{equation}
and  
\begin{equation*}
\mathcal{H}(\omega,\rho,v,\omega_N,Q_N,P_N) \geq 0.
\end{equation*}
\end{proposition}

Remark that if we remove $B N^{-\gamma}$ and if we set $\omega_N = \omega = 0$ our quantity $\mathcal{H}$ is the functionnal used by Duerinckx in the appendix of \cite{Serfaty} to prove the mean-field limit of particles satisfying \eqref{edo_generique_ordre_2} to the Euler-Poisson equations.

Our main result is the following theorem:
\begin{theorem}
\label{theoreme_champ_moyen}
Let $(\rho,\omega,v)$ be a weak solution of System \eqref{edp} in the sense of Definition \ref{def_sol_faible_edp} and $(\omega_N,Q_N,P_N)$ be a weak solution of System \eqref{edo} in the sense of Definition \ref{def_solution_faible_edo}. Then we define
\begin{equation}
\label{definition_H_N}
\mathcal{H}_N(t) := \mathcal{H}(\omega(t),\rho(t),v(t),\omega_N(t),Q_N(t),P_N(t)).
\end{equation}
Suppose that $\nabla \omega \in L^\infty$, $\nabla v \in \mathcal{C}^{0}([0,T]\times\mathbb{R}^2,\mathbb{R}^2)$ and that
\begin{align}
\underset{N \in \mathbb{N}}{\sup}\norm{\omega_N^0}_{L^1 \cap L^\infty} &< +\infty \label{controle_uniforme_omega_N_0}\\ 
q_1(0),...,q_N(0) &\notin \supp(\omega_N^0) \label{points_hors_du_support} \\
\forall i \neq j, q_i(0) \neq q_j(0). \label{points_distincts}
\end{align}
Then there exist positive constants C and $\beta$ depending only on $T,\rho,\omega$ and $\norm{\omega_N}_{L^\infty}$ such that for all $t \in [0,T]$,
\begin{equation}
\label{controle_energie_modulee}
\mathcal{H}_N(t) \leq C(\mathcal{H}_N(0) + N^{-\beta}).
\end{equation}
\end{theorem}

\begin{remark}
By Sobolev embeddings the solutions of the spray system \eqref{edp} given by Theorem \ref{main_theorem_well_posedness} are also solutions in the sense of Definition \ref{def_sol_faible_edp} that satisfy the hypothesis of Theorem \ref{theoreme_champ_moyen} and thus Theorem \ref{main_theorem_well_posedness} gives the existence of sufficiently regular solutions of System \eqref{edp} that can be approached as mean-field limits of solutions of System \eqref{edo} (even if Theorem \ref{theoreme_champ_moyen} does not recquire solutions to be as regular as the solutions obtained in Theorem \ref{main_theorem_well_posedness}).
\end{remark}

We will also prove a coerciveness result about this energy.
\begin{proposition}
\label{proposition_coerciveness_energy}

Let $Q_N,P_N \in (\mathbb{R}^2)^N$ and let $\omega,\omega_N,\rho \in L^1\cap L^\infty(\mathbb{R}^2,\mathbb{R})$ be probability densities  with compact supports and $v \in W^{2,\infty}(\mathbb{R}^2,\mathbb{R}^2)$. Assume that
\begin{equation}\label{controle_uniforme_omega_N_coerc}
\underset{N \in \mathbb{N}}{\sup}\norm{\omega_N}_{L^\infty} < + \infty.
\end{equation} 
Then there exist positive constants $C$ and $\beta$ such that
\begin{multline}
\label{coerciveness_energy}
\norm{\frac{1}{N}\sum_{i=1}^N \delta_{(q_i,p_i)} - \rho\otimes\delta_{\xi = v(x)}}_{H^{-5}} \leq C(1+\norm{\nabla v}_{W^{1,\infty}}^2) \\ \times(\mathcal{H}(\omega,\rho,v,\omega_N,Q_N,P_N)^\frac{1}{2} + (1+\norm{\rho}_{L^\infty})N^{-\beta})^\frac{1}{2}.
\end{multline}
In particular, if we assume that
\begin{equation*}
\mathcal{H}(\omega,\rho,v,\omega_N,Q_N,P_N) \underset{N \rightarrow \infty}{\longrightarrow} 0
\end{equation*}
then for any $a < -1$,
\begin{align*}
\rho_N &\Tend{N}{+\infty} \rho  \; \; \text{in} \; H ^{a} \\
\omega_N - \omega &\Tend{N}{+\infty} 0 \; \; \text{in} \; L^2\cap \dot H^{-1} \\
\frac{1}{N}\sum_{i=1}^N \delta_{(q_i,p_i)} &\Tend{N}{+\infty} \rho\otimes\delta_{\xi = v(x)} \; \; \text{in} \; H ^{-5}.
\end{align*}
\end{proposition}

\begin{remark} The $H^{-5}$ norm is not optimal, but it is sufficient to justify that $\mathcal{H}_N$ controls the convergence to a monokinetic distribution in a weak sense. 
\end{remark} 

As a consequence we get that if a sequence of solutions $(\omega_N,Q_N,P_N)$ of \eqref{edo} satisfying the hypothesis of Theorem \ref{theoreme_champ_moyen} are such that
\begin{equation*}
\mathcal{H}_N(0) \Tend{N}{+\infty} 0
\end{equation*}
then for any $t \in [0,T]$ we have
\begin{equation*}
\mathcal{H}_N(t) \Tend{N}{+\infty} 0
\end{equation*}
and it follows by Proposition \ref{proposition_coerciveness_energy} that for any $t \in [0,T]$ and $a < -1$,
\begin{align*}
\rho_N(t) &\Tend{N}{+\infty} \rho(t)  \; \; \text{in} \; H ^{a}  \\
\omega_N(t) - \omega(t) &\Tend{N}{+\infty} 0 \; \; \text{in} \; L^2\cap \dot H^{-1} \\
\frac{1}{N}\sum_{i=1}^N \delta_{(q_i(t),p_i(t))} &\Tend{N}{+\infty} \rho(t)\otimes\delta_{\xi = v(t,x)} \; \; \text{in} \; H ^{-5}.
\end{align*}
Since $\displaystyle{\frac{1}{N}\sum_{i=1}^N\delta_{(q_i(t),p_i(t))}}$ is bounded in the dual of continuous bounded functions, we can extract a subsequence which will converge in the weak-$\ast$ topology of signed measures $\mathcal{M}(\mathbb{R}^2\times\mathbb{R}^2)$. Since it necessarily converges to $\rho(t)\otimes \delta_{\xi = v(t,x)}$, by weak-$\ast$ compactness we can deduce that for all $t \in [0,T]$,
\begin{equation*}
\frac{1}{N}\sum_{i=1}^N\delta_{(q_i(t),p_i(t))} \overset{\ast}{\rightharpoonup} \rho(t)\otimes\delta_{\xi = v(t,x)} \; \; \text{in} \; \mathcal{M}(\mathbb{R}^2)
\end{equation*}
and thus we have the mean-field convergence of \eqref{edo} to \eqref{edp}. One can look at \cite[Corollary 1.2]{Rosenzweig} for a more detailed proof of such a compactness argument.

\begin{remark}
If we suppose that $\omega_{N,0} - \omega_0$ converges to 0 in $\dot H^{-1}$ and that $\rho_{N,0}$ converges to $\rho_0$ in the weak-$*$ topology of signed measures, then the convergence of
\begin{equation*}
\int_{(\mathbb{R}^2\times\mathbb{R}^2)\backslash \Delta} g(x-y)(\rho_0 + \omega_0 - \rho_{N,0} - \omega_{N,0})^{\otimes 2}(\dd x\dd y)
\end{equation*}
to zero can be seen as a well-preparedness condition on the initial data, as stated in the following proposition:
\end{remark}

\begin{proposition}\label{proposition_donnees_initiales_bien_preparees}
Let us suppose that $\omega_{N,0}, \omega_0, \rho_0 \in L^2(\mathbb{R}^2,\mathbb{R})$ are probability densities with compact support and that $(q_1^0,...,q_N^0)$  are such that $q_i^0 \neq q_j^0$ if $i \neq j$.  Then if we suppose
\begin{equation}\label{controle_uniforme_omega_N_0_well_prep_init_data}
\underset{N \in \mathbb{N}}{\sup}\norm{\omega_N^0}_{L^1 \cap L^\infty} < +\infty 
\end{equation}
and
\begin{align*}
\omega_{N,0} - \omega_0 & \Tend{N}{+\infty} 0 \; \; \text{in} \; \dot H^{-1}  \\
\rho_{N,0} &\underset{N\longrightarrow +\infty}{\overset{\ast}{\rightharpoonup}} \rho_0 \; \; \text{in} \; \mathcal{M}(\mathbb{R}^2) \\
\frac{1}{N^2}\sum_{1 \leq i \neq j \leq N} g(q_i^0 - q_j^0) &\Tend{N}{+\infty} \iint g(x-y)\rho_0(x)\rho_0(y)\dd x \dd y
\end{align*}
we have
\begin{equation*}
\int_{(\mathbb{R}^2\times\mathbb{R}^2)\backslash \Delta} g(x-y)(\rho_0 + \omega_0 - \rho_{N,0} - \omega_{N,0})^{\otimes 2}(\dd x\dd y) \Tend{N}{+\infty} 0.
\end{equation*}
\end{proposition}

The latter statement strongly relies on the results proved in \cite{Duerinckx}. One could have more details about these well-preparedness assumptions by reading the introduction of \cite{Duerinckx}.

The remainder of this paper is organized as follows. In Section \ref{section:2} we establish local well-posedness of strong solutions of \eqref{edp}. Then in Section \ref{section:3} we provide the proof of Proposition \ref{asmptotic_positivity_energy}, Theorem \ref{theoreme_champ_moyen}, Proposition \ref{proposition_coerciveness_energy} and Proposition \ref{proposition_donnees_initiales_bien_preparees}. Sections \ref{section:2} and \ref{section:3} are independent of each other.

\section{Local Well-Posedness}\label{section:2}

In this section, if $\mu$ is a continuous function defined on $\mathbb{R}^2$ with compact support, we will denote 
\begin{equation*}
R[\mu] := \text{sup} \; \{|x|  \; ;  \; x \in \mathbb{R}^2, \mu(x) \neq 0\}
\end{equation*}
and
\begin{equation*}
R_T[\mu] := \underset{0 \leq t \leq T}{\text{sup}} \; R[\mu(t)]
\end{equation*}
if $\mu$ depends on time. If B is a Banach space and $1 \leq p\leq \infty$, we will denote
\begin{equation*}
L^p_T B := L^p([0,T],B).
\end{equation*}
We will use the same convention for the Hölder spaces $\mathcal{C}^k_T B$ and the Sobolev spaces $W^{k,p}_T B$. Let us also recall that $g$ is the opposite of the Green kernel on the plane defined in \eqref{definition_g}.

C will refer to a constant independent of time and of any other parameter that can change value from one line to another. We will denote $C(A,B)$ for a constant depending only on some quantities A and B.

We want to show that System \eqref{edp} has a unique regular solution on $[0,T]$ for $T$ small enough. In \cite{Makino}, Makino builds such a solution for the following compressible Euler-Poisson system in three dimensions:
\begin{equation*}
\left\{
\begin{aligned}
&\partial_t u + (u\cdot\nabla) u + \nabla p = F\ast\rho \\
&\partial_t \rho + \div(\rho u) = 0
\end{aligned}\right.
\end{equation*} 
where $p$ is a function of $\rho$ and $F := \nabla G$ where $G$ is the Green function on $\mathbb{R}^3$. There are three main differences with our system \eqref{edp}:
\begin{enumerate}

\item We have no pressure term, but we have a gyroscopic effect.
\item We have a continuity equation on $\omega$ that we also need to solve.
\item On the plane $\mathbb{R}^2$, the function $V = -\nabla^\bot g\ast(\rho + \omega)$ is not in $L^2$ except if we assume that $\displaystyle{\int (\rho + \omega) = 0}$.
\end{enumerate}

In order to deal with the third point, we will assume that $v_0 = u_0 + \overline{V}$ where $u_0 \in L^2$ and $\overline{V}$ is a function of $x$ that we will specify later. If we try to find a solution of \eqref{edp} when $v = u + \overline{V}$, we find that $(u,\rho,\omega)$ evolves according to the following equations:
\begin{equation}
\label{euler_poisson_l2}
\left\{
\begin{aligned}
& \partial_t \omega + \div(\omega V) = 0 \\
& \partial_t \rho + \div(\rho v) = 0 \\
& \partial_t u + ((u + \overline{V})\cdot\nabla)u + (u\cdot\nabla) \overline{V} = u^\bot + f \\
& V = -\nabla^\bot g\ast(\omega + \rho) \\
& f = (\overline{V}-V)^\bot - (\overline{V}\cdot\nabla)\overline{V} \\
& v = u + \overline{V}.
\end{aligned}\right.
\end{equation} 

Thus if we choose $\overline{V}$ such that $f \in L^2$, we will find an equation that we expect to have a solution in $L^2$. We can achieve this goal choosing the following value of $\overline{V}$:
\begin{equation*}
\overline{V} := -\left(\int_{\mathbb{R}^2}\omega_0 + \rho_0\right)\nabla^\bot g\ast\chi
\end{equation*}
where $\chi$ is some compactly supported function such that $\displaystyle{\int_{\mathbb{R}^2} \chi = 1}$. We make such a choice because $\displaystyle{\int_{\mathbb{R}^2}\rho}$ and $\displaystyle{\int_{\mathbb{R}^2}\omega}$ are conserved and we will justify later that for $\mu$ compactly supported,
\begin{equation*}
-\nabla^\bot g \ast \mu(x) = \frac{1}{2\pi}\left(\int_{\mathbb{R}^2}\mu\right)\frac{x^\bot}{|x|^2} + \underset{|x|\rightarrow\infty}{\mathcal{O}}(|x|^{-2}).
\end{equation*}
Since we assumed that $\rho$ and $\omega$ have compact support, we are not concerned by the fact that $V$ is not $L^2$ on the whole plane. Remark also that the space
\begin{equation*}
-\left(\int_{\mathbb{R}^2}\omega_0 + \rho_0\right)\nabla^\bot g\ast\chi + L^2
\end{equation*}
does not depend on the choice of $\chi$. Now we are able to write the main theorem of this section:
\begin{theorem}\label{main_theorem_well_posedness}
Let $s$ be an integer such that $s \geq 3$ and $u_0 \in H^{s+1}(\mathbb{R}^2,\mathbb{R}^2)$, $\rho_0,\omega_0 \in H^{s}(\mathbb{R}^2,\mathbb{R})$ such that $\omega_0$ and $\rho_0$ have compact support, then if $T$ is small enough (with respect to some quantity depending only on $\norm{\omega_0}_{H^s}$, $\norm{\rho_0}_{H^s}$, $\norm{u_0}_{H^{s+1}}$, $R[\omega_0]$ and $R[\rho_0]$), there exists a unique $(u,\omega,\rho)$  with $(\rho,\omega) \in (C_T H^s \cap C^1_T H^{s-1})^2$ and $u \in C_T H^{s+1} \cap C^1_T H^s$ solution of \eqref{euler_poisson_l2}.
\end{theorem}

The proof of Theorem \ref{main_theorem_well_posedness} proceeds as follows:
\begin{enumerate}

\item We fix $T > 0$ and define
\begin{align*}
R_0 &:= R[\rho_0 + \omega_0] \\
M_0 &:= \max(\norm{\rho_0}_{H^s},\norm{\omega_0}_{H^s},\norm{u_0}_{H^{s+1}})
\end{align*}
\begin{equation}\label{definition_X_T}
\begin{aligned}
X_T := \; \bigg\{ &(\omega,\rho) \in  L^\infty_T H^s \cap C_T H^{s-1} \bigg | \omega(0) = \omega_0, \rho(0) = \rho_0, \\
&\norm{\rho}_{L^\infty_T H^s} \leq 2M_0, \norm{\omega}_{L^\infty_T H^s} \leq 2M_0, R_T[\rho + \omega] \leq 2 R_0, \\
&\forall t \in [0,T], \int(\rho(t)+\omega(t)) = \int (\rho_0+\omega_0), \\
&\forall t,t' \in [0,T], \norm{\rho(t) - \rho(t')}_{H^{s-1}} \leq L|t-t'|, \\ 
&\norm{\omega(t) - \omega(t')}_{H^{s-1}} \leq L|t-t'|, \bigg\}
\end{aligned}
\end{equation}
where $L > 0$ is a quantity depending only on $R_0$ and $M_0$. Remark that $X_T$ is a subspace of $(C_T H^s \cap C^1_T H^{s-1})^2$. Then we fix $(\omega,\rho) \in X_T$ and we define
\begin{align}
V &:= -\nabla^\bot g\ast(\rho+\omega)\label{definition_V} \\
\overline{V} &:= -\left(\int_{\mathbb{R}^2}\omega_0 + \rho_0\right)\nabla^\bot g\ast\chi\label{definition_V_bar} \\
f &:= (\overline{V}-V)^\bot - \overline{V}\cdot\nabla\overline{V}.\label{definition_f}
\end{align}
Note that we will prove in Subsection \ref{subsection_BS} that $f \in \mathcal{C}_T H^s\cap L^\infty_T H^{s+1}$.

\item In Subsection \ref{subsection_presureless_euler} we solve in $C_{T_1} H^{s+1} \cap C^1_{T_1} H^s$ the equation
\begin{equation*}
\partial_t u + ((u + \overline{V})\cdot\nabla) u + (u\cdot\nabla)\overline{V} = u^\bot + f
\end{equation*}
for initial condition $u_0$ and $T_1 \leq T$ small enough depending only on $M_0$ and $L$.

\item In Subsection \ref{subsection_continuity_equations} we define $v = u + \overline{V}$ and solve in $(C_{T_1} H^s \cap C^1_{T_1} H^{s-1})^2$ the system
\begin{equation*}
\left\{
\begin{aligned}
& \partial_t \widetilde \omega + \div(\widetilde \omega V) = 0 \\
& \partial_t \widetilde \rho + \div(\widetilde \rho v) = 0.
\end{aligned}\right.
\end{equation*} 

\item In Subsection \ref{subsection_EEP} we apply a fixed-point theorem by showing that the map defined on $X_{T_2}$ by $(\omega,\rho) \longmapsto (\widetilde \omega,\widetilde \rho)$ is a contraction for the $\mathcal{C}_T L^2$ norm if $T_2 \leq T_1$ is small enough, using the estimates proved for the previous equations.
\end{enumerate}

\begin{remark}
$X_T$ is strictly included in $(C_{T_1} H^s \cap C^1_{T_1} H^{s-1})^2$ but since we prove in step $(3)$ that the image of the application $\Phi$ sending $(\omega,\rho)$ to $(\widetilde{\omega},\widetilde{\rho})$ is contained in $(C_{T_1} H^s \cap C^1_{T_1} H^{s-1})^2$ we have the expected regularity for the solutions of our system.
\end{remark}

\begin{remark}
Uniqueness is established in the space $X_T$ which is bigger than $(C_T H^s \cap C^1_T H^{s-1})^2$. It ensures uniqueness for the whole system: If $(\rho_1,\omega_1,u_1)$ and $(\rho_2,\omega_2,u_2)$ are two solutions of \eqref{edp}, then $(\rho_1,\omega_1) = (\rho_2,\omega_2)$ by uniqueness of the fixed point and $u_1 = u_2$ follows by uniqueness of solutions of Equation \eqref{euler_equation_V_bar}. Remark also that using energy estimates one could prove uniqueness in a space of smaller regularity. 
\end{remark}

Before doing these different steps, we give some results about the Biot-Savart kernel $-\nabla^\bot g$ that we will need later. In this section we will use the following definition of uniformly local Sobolev spaces:
\begin{definition}
We define $H^s_{\rm ul}(\mathbb{R}^2)$ as the space of locally $H^s$ functions verifying
\begin{equation*}
\norm{u}_{H^s_{\rm ul}(\mathbb{R}^2)} := \underset{x \in \mathbb{R}^2}{\sup}\norm{u}_{H^s(B(x,1))} < +\infty.
\end{equation*}
\end{definition}
For a more complete introduction to these spaces we refer to Section 2.2 of \cite{Kato}.

\subsection{Properties of the Biot-Savart kernel on the plane} \label{subsection_BS}

In this subsection we prove Proposition \ref{proposition_biot_savart}, which contains several results about the Biot-Savart kernel $-\nabla^\bot g$.
\begin{proposition}\label{proposition_biot_savart}
Let $s \geq 3$ and let $\mu$ be a $H^s$ function on $\mathbb{R}^2$ with compact support. Denote 
\begin{equation*}
V := -\nabla^\bot g \ast \mu.
\end{equation*}
Then we have the following inequalities:
\begin{enumerate}

\item $\displaystyle{V \in H^{s+1}_{\rm{ul}}}$ and $\displaystyle{\norm{V}_{H^{s+1}_{\rm{ul}}} \leq C(1+R[\mu])\norm{\mu}_{H^s}}$.

\item $\displaystyle{\norm{\nabla V}_{H^s} \leq C\norm{\mu}_{H^s}}$.

\item $V \in L^\infty$ and we have the three following bounds:
\begin{align*}
\norm{V}_{L^\infty} &\leq CR[\mu]\norm{\mu}_{L^\infty} \\
\norm{V}_{L^\infty} &\leq CR[\mu]^\frac{1}{3}\norm{\mu}_{H^1} \\
\norm{V}_{L^\infty} &\leq C\norm{\mu}_{L^1}^\frac{1}{2}\norm{\mu}_{L^\infty}^\frac{1}{2}.
\end{align*}

\item $\displaystyle{\norm{(V\cdot\nabla) V}_{H^s} \leq C(1+R[\mu])\norm{\mu}_{H^s}^2}$.

\item $\displaystyle{V(x) = \frac{1}{2\pi}\bigg(\int_{\mathbb{R}^2} \mu\bigg)\frac{x^\bot}{|x|^2} + \underset{|x|\rightarrow\infty}{\mathcal{O}}(|x|^{-2})}$.

\item If $\mu$ has mean zero, then $V \in L^2$ and
\begin{equation*}
\norm{V}_{L^2} \leq C(R[\mu]+R[\mu]^3)^\frac{1}{2}\norm{\mu}_{L^2}.
\end{equation*}

\end{enumerate}

\end{proposition}

Estimates $(1)$ to $(4)$ are consequences of the two following propositions. The first one is the usual potential estimate of a velocity field given by the Biot-Savart law: 
\begin{proposition}[Potential estimates in $L^p$]\label{potential_estimates}
If $2 < p < \infty$ and $\omega \in L^1 \cap L^p$, then
\begin{equation*}
\norm{\nabla g*\omega}_{L^\infty} \leq C_p\norm{\omega}_{L^1}^\frac{p-2}{2p-2}\norm{\omega}_{L^p}^\frac{p}{2p-2}
\end{equation*}
\begin{equation*}
\norm{\nabla g*\omega}_{L^\infty} \leq C\norm{\omega}_{L^1}^\frac{1}{2}\norm{\omega}_{L^\infty}^\frac{1}{2}.
\end{equation*}
\end{proposition}
For the proof of this proposition see for example \cite[Lemma 1]{Iftimie}. The second is the Calderón-Zygmund inequality:
\begin{proposition}[Calderón-Zygmund inequality]\label{Calderon_Zygmund}
If $1 < p < +\infty$, 
\begin{equation*}
\norm{\nabla^2 g*\omega}_{L^p} \leq C_p \norm{\omega}_{L^p}.
\end{equation*}
\end{proposition}
For the proof of this inequality we refer to \cite[Proposition 7.5]{BahouriCheminDanchin}.

Claims $(5)$ and $(6)$ giving the behavior of V at infinity are classical results in fluid dynamics (see for example \cite[Proposition 3.3]{MajdaBertozzi}) that we will prove to have the specific $L^2$ bound we need on $V$.

The first consequence of Proposition \ref{proposition_biot_savart} is the following:
\begin{corollary}\label{corollaire_controle_f_V_bar}
Let $\rho_0,\omega_0 \in H^s$ with compact support, $\chi$ be a smooth function with compact support and $(\rho,\omega) \in X_T$ where $X_T$ is the space defined by \eqref{definition_X_T}. Let us consider the functions $\overline{V}$ and $f$ defined by \eqref{definition_V_bar} and \eqref{definition_f}, then we have  $\overline V \in H^{s+2}_{\rm ul}$, $f \in L^\infty_TH^{s+1}\cap \mathcal{C}_T H^s$ and
\begin{align*}
\norm{\overline V}_{H^{s+2}_{\rm ul}} &\leq C(R_0,M_0) \\
\norm{f}_{L^\infty_TH^{s+1}} &\leq C(R_0,M_0).
\end{align*}
\end{corollary}
\begin{proof}[Proof of Proposition \ref{proposition_biot_savart}]

Let us begin by the second inequality. We have:
\begin{equation*}
\norm{\nabla V}_{H^{s}} = \norm{\nabla^2 g\ast\mu}_{H^{s}} \leq C\sum_{|\alpha| \leq s} \norm{\nabla^2 g\ast\partial^\alpha\mu}_{L^ 2} \leq C\norm{\mu}_{H^s}
\end{equation*}
by Proposition \ref{Calderon_Zygmund}.

Let us now prove the third Claim. By Proposition \ref{potential_estimates}, we have
\begin{align*}
\norm{V}_{L^\infty} 
&\leq C\norm{\mu}_{L^1}^\frac{1}{2}\norm{\mu}^\frac{1}{2}_{L^\infty} \\
&\leq C\norm{\mu}_{L^\infty}^\frac{1}{2}\left(\int_{B(0,R[\mu])}|\mu|\right)^\frac{1}{2}  \\
&\leq C\norm{\mu}_{L^\infty}\left(\int_{B(0,R[\mu])}1\right)^\frac{1}{2} \\
&\leq CR[\mu]\norm{\mu}_{L^\infty}.
\end{align*}
For the second inequality of $(3)$, we use Proposition \ref{potential_estimates} again to get
\begin{align*}
\norm{V}_{L^\infty} 
&\leq C\norm{\mu}_{L^1}^\frac{1}{3}\norm{\mu}_{L^4}^\frac{2}{3}.
\end{align*}
Moreover, by Cauchy-Schwartz inequality,
\begin{align*}
\norm{\mu}_{L^1} \leq C\norm{\mathbf{1}_{B(0,R[\mu])}}_{L^2}\norm{\mu}_{L^2} \leq C R[\mu]\norm{\mu}_{H^1}.
\end{align*}
and therefore by the embedding of $H^1$ into $L^4$ (see for example \cite[Corollary 9.11]{Brezis}) we have
\begin{align*}
\norm{V}_{L^\infty} \leq CR[\mu]^\frac{1}{3}\norm{\mu}_{H^1}.
\end{align*}
The third inequality of $(3)$ is the second inequality of  Proposition \ref{potential_estimates}.

The first inequality follows from the two Claims we just proved: Since all derivatives of $V$ of order $k$ for $1 \leq k \leq s+1$ belong to $L^2$ and since 
$\norm{V}_{L^2_{\rm{ul}}} \leq C\norm{V}_{L^\infty}$, we get
\begin{align*}
\norm{V}_{H^{s+1}_{\rm{ul}}} 
&\leq C(\norm{\mu}_{H^s} + R[\mu]\norm{\mu}_{L^\infty}) \\
&\leq C(1+R[\mu])\norm{\mu}_{H^s}
\end{align*}
because $H^s \hookrightarrow L^\infty$.

Now let us prove the fourth point. Let $\alpha$ be a multi-index such that $|\alpha| \leq s$, then $\partial^\alpha ((V \cdot\nabla) V)$ is a combination of $(\partial^{\alpha_1}V \cdot\nabla)\partial^{\alpha_2} V$ where $\alpha_1 + \alpha_2 = \alpha$. If $\alpha_1 = 0$,
\begin{align*}
\norm{(\partial^{\alpha_1} V \cdot \nabla)\partial^{\alpha_2} V}_{L^2} &\leq \norm{V}_{L^\infty}\norm{\nabla V}_{H^s}.
\end{align*}

If $1 \leq |\alpha_1| \leq s-1$, then
\begin{align*}
\norm{(\partial^{\alpha_1} V \cdot \nabla)\partial^{\alpha_2} V}_{L^2} &\leq \norm{\partial^{\alpha_1}V}_{L^\infty}\norm{\nabla V}_{H^s} \\
&\leq \norm{\nabla V}_{H^s}^2.
\end{align*}

Finally if $|\alpha_1| = s$,
\begin{align*}
\norm{(\partial^{\alpha_1} V \cdot \nabla)\partial^{\alpha_2} V}_{L^2} &\leq \norm{\partial^{\alpha_1}V}_{L^2}\norm{\nabla V}_{L^\infty} \\
&\leq \norm{\nabla V}_{H^s}^2.
\end{align*}

We conclude using (2) and (3).

We now prove the fifth claim by a standard argument. Let us set
\begin{equation*}
W(x + iy) = V_1(x,y) - iV_2(x,y).
\end{equation*}
Then we have
\begin{align*}
(\partial_x + i \partial_y)W 
&= (\partial_x V_1 + \partial_y V_2) + i(\partial_y V_1 - \partial_x V_2) \\
&= \div(V) - i \curl(V) \\
&= 0 - i\mu.
\end{align*}

Thus $W$ is holomorphic on $\mathbb{C}\backslash B(0,R)$ with $R = R[\mu]$ (since it is a solution of Cauchy-Riemann equations) and we can write it as the sum of a Laurent serie:
\begin{equation*}
W(z) = \sum_{k=-\infty}^{+\infty}a_k z^{-k}.
\end{equation*}
Remark that since we have $W(z) \Tend{z}{\infty} 0$, $a_k = 0$ for $k$ nonpositive. Now we compute $a_1$ by a contour integral in the counter clockwise sense:
\begin{align*}
a_1 &= \frac{1}{2i\pi}\int_{\partial B(0,R)}W(z)\dd z \\
&= \frac{1}{2i\pi}\int_0^{2\pi}W(Re^{i\theta})Rie^{i\theta}\dd\theta \\
&= \frac{1}{2\pi}\int_0^{2\pi} (V_1-iV_2)(R\cos(\theta),R\sin(\theta))(R\cos(\theta)+iR\sin(\theta))\dd\theta \\
&= \frac{1}{2\pi} \int_{\partial B(0,R)}(V\cdot n + i V^\bot \cdot n)\dd\sigma
\end{align*}
$n$ beeing the outer normal vector to $B(0,R)$ (or equivalently the inner normal vector to $B(0,R)^c$) and $\sigma$ its unit measure. Thus by Stokes theorem,
\begin{align*}
a_1 &= \frac{1}{2\pi}\int_{B(0,R)}\div(V) + \frac{i}{2\pi}\int_{B(0,R)}\div(V^\bot) \\
&= -\frac{i}{2\pi}\int_{\mathbb{R}^2}\curl(V) \\
&= -\frac{i}{2\pi}\int_{\mathbb{R}^2}\mu.
\end{align*}
Finally, we get
\begin{align*}
V_1 + i V_2 &= \frac{i}{2\pi}\bigg(\int_{\mathbb{R}^2} \mu\bigg)\frac{1}{\overline z} + O(|z|^{-2})\\
&= \frac{1}{2\pi}\bigg(\int_{\mathbb{R}^2} \mu\bigg)i\frac{z}{|z|^2} + O(|z|^{-2})
\end{align*}
which gives us the fifth claim.

Now let us assume that $\displaystyle{\int_{\mathbb{R}^2} \mu = 0}$ and bound the $L^2$ norm of $V$. Let $\psi = g\ast\mu$, then by the fifth point of the inequality, $W$ is holomorphic in $B(0,R)^c$ and has a holomorphic primitive. Thus we get $\psi(x) = D + O(|x|^{-1})$ and for $A > 0$ big enough,
\begin{align*}
\left|\int_{\partial B(0,A)} V\psi\right| 
&\leq C\norm{V}_{L^\infty(\partial B(0,A))}\norm{\psi}_{L^\infty(\partial B(0,A))}2\pi A  \\
&\leq \frac{CD}{A} \Tend{A}{+\infty} 0.
\end{align*}

This fact allows us to compute the following integral by parts:
\begin{align*}
\int_{\mathbb{R}^2}  |V|^2 &= \int_{\mathbb{R}^2} \nabla \psi \cdot \nabla \psi \\
&= -\int_{\mathbb{R}^2} \psi \Delta \psi \\
&= \int_{\mathbb{R}^2} \psi \mu \\
&\leq \norm{\psi}_{L^\infty(\supp(\mu))}\norm{\mu}_{L^1} \\
&\leq CR[\mu]\norm{\mu}_{L^2}\norm{\psi}_{L^\infty(\supp(\mu))}.
\end{align*}
Now if $x \in B(0,R[\mu])$, we have
\begin{align*}
|\psi(x)| \leq& \; C\bigg(- \int_{B(x,1)\cap B(0,R[\mu])}\ln(|x-y|)|\mu(y)|\dd y \\ 
&+ \int_{B(x,1)^c\cap B(0,R[\mu])}(|x|+|y|)|\mu(y)|\dd y\bigg) \\
\leq& \; C\bigg(- \int_{B(x,1)}\ln(|x-y|)|\mu(y)|\dd y \\ 
&+ \int_{B(0,R[\mu])}(2R[\mu])|\mu(y)|\dd y\bigg) \\
\leq& \; C(1+R[\mu]^2)\norm{\mu}_{L^2}.
\end{align*}
Thus
\begin{align*}
\int |V|^2 &\leq CR[\mu](1+R[\mu]^2)\norm{\mu}^2_{L^2}
\end{align*}
which is the sixth claim of our proposition.
\end{proof}

Now we prove the uniform bounds we need on $f$ and $\overline{V}$:
\begin{proof}[Proof of Corollary \ref{corollaire_controle_f_V_bar}]
First remark that
\begin{align*}
\norm{\overline V}_{H^{s+2}_{\rm ul}}
&= \left|\int_{\mathbb{R}^2} \rho_0 + \omega_0\right|\norm{\nabla g \ast \chi}_{H^{s+2}_{\rm ul}} \\
&\leq C\left|\int \rho_0 + \omega_0\right| \\
&\leq C\norm{\rho_0 + \omega_0}_{L^1} \\
&\leq CR[\rho_0+\omega_0]\norm{\rho_0 + \omega_0}_{L^2} \\
&\leq 2CR_0M_0
\end{align*}
by Claims $(1)$ and $(3)$ of Proposition \ref{proposition_biot_savart}. Moreover, if we denote
\begin{equation*} 
h = \rho + \omega - \left(\int_{\mathbb{R}^2}\rho_0 + \omega_0\right)\chi
\end{equation*}
we have
\begin{align*}
\norm{f}_{L^\infty_T H^{s+1}} 
\leq& \norm{(V - \overline{V})^\bot - (\overline{V}\cdot \nabla) \overline{V}}_{L^\infty_T H^{s+1}} \\
\leq& \norm{(V - \overline{V})^\bot}_{L^\infty_TL^2} + \norm{\nabla(V - \overline{V})^\bot}_{L^\infty_T H^s} \\
&+ \norm{(\overline{V}\cdot \nabla) \overline{V}}_{L^\infty_T H^{s+1}} \\
\leq& \norm{\nabla g\ast h}_{L^\infty_TL^2} + \norm{\nabla^2 g\ast h}_{L^\infty_T H^s} + \norm{(\overline{V} \cdot\nabla) \overline{V}}_{L^\infty_TH^{s+1}}  \\
\leq&  C(R_T[h]+R_T[h]^3)^\frac{1}{2}\norm{h}_{L^\infty_TL^2} +  C\norm{h}_{L^\infty_TH^s} \\
&+ C(1+R_T[\chi])\left|\int\rho_0 + \omega_0\right|^2\norm{\chi}^2_{L^\infty_T H^{s+1}} \\
\leq& C(R_0,M_0)
\end{align*}
where we used Claims $(2)$, $(4)$ and $(6)$ of Proposition \ref{proposition_biot_savart}.

Now let us justifiy that $f \in \mathcal{C}_T H^s$. If $t_1,t_2 \in [0,T]$, we have
\begin{align*}
\norm{f(t_1) - f(t_2)}_{H^s}
=& \; \norm{\nabla^\bot g\ast(\rho(t_1) + \omega(t_1) - \rho(t_2) - \omega(t_2))}_{H^s} \\
\leq& \; \norm{\nabla^\bot g\ast(\rho(t_1) + \omega(t_1) - \rho(t_2) - \omega(t_2))}_{L^2} \\
&+ \norm{\nabla^2 g\ast(\rho(t_1) + \omega(t_1) - \rho(t_2) - \omega(t_2))}_{H^{s-1}} \\
\leq& \; C(R_T[\rho+\omega]+R_T[\rho+\omega]^3)^\frac{1}{2} \\
&\times\norm{\rho(t_1) + \omega(t_1) - \rho(t_2) - \omega(t_2)}_{L^2} \\
&+ C\norm{\rho(t_1) + \omega(t_1) - \rho(t_2) - \omega(t_2)}_{H^{s-1}}
\end{align*}
where we used points $(6)$ and $(2)$ of Proposition \ref{proposition_biot_savart} and therefore $f \in\mathcal{C}_T H^s$ follows from $\rho, \omega \in \mathcal{C}_TH^{s-1}$.
\end{proof}

\subsection{Pressureless Euler equations} \label{subsection_presureless_euler}

In this subsection we prove that there is a unique solution to the following equation
\begin{equation}
\label{euler_equation_V_bar}
\partial_t u + ((u + \overline{V})\cdot \nabla) u + (u\cdot\nabla) \overline{V} = u^\bot + f
\end{equation}
where $\overline{V}$ and $f$ are the functions defined in \eqref{definition_V_bar} and \eqref{definition_f}.

Following the idea of \cite{Kato,Makino}, we start by fixing $u \in \mathcal{C}_T H^{s+1}$ and solving the linearized equation:
\begin{equation}
\label{linear_euler_equation_V_bar}
\left\{
\begin{aligned}
& \partial_t \widetilde u + ((u + \overline{V})\cdot \nabla) \widetilde u + (\widetilde u\cdot\nabla) \overline{V} = \widetilde u^\bot + \widetilde f \\
& \widetilde{u}(0) = \widetilde{u}_0.
\end{aligned}\right.
\end{equation}
We have the following well-posedness theorem:
\begin{theorem}\label{well_posedness_linear_presureless_euler}
If $s$ is an integer such that $s \geq 3$, $u \in \mathcal{C}_T H^{s+1}$, $\widetilde u_0 \in H^{s+1}$, $\mu$ with compact support and $\widetilde f \in L^1_T H^{s+1} \cap \mathcal{C}_T H^s$, then \eqref{linear_euler_equation_V_bar} has a solution $\widetilde u \in \mathcal{C}_T H^{s+1} \cap \mathcal{C}^1_T H^s$, unique in the space $\mathcal{C}_T H^1 \cap \mathcal{C}^1_T L^2$. Moreover, we have the following estimates:
\begin{align*}
\norm{\widetilde u(t)}_{H^{s+1}} &\leq e^{CT(\norm{\overline{V}}_ {L^\infty_T H^{s+2}_{\rm{ul}}} + \norm{u}_{L^\infty_T H^{s+1}}+1)} \left(\norm{\widetilde u_0}_{H^{s+1}} + C\norm{\widetilde f}_{L^1_T H^{s+1}}\right) \\
\norm{\partial_t\widetilde u(t)}_{H^s} &\leq C\left(\norm{\widetilde f(t)}_{H^s} + (\norm{u}_{L^\infty_T H^{s+1}}+\norm{\overline{V}}_ {L^\infty_T H^{s+2}_{\rm{ul}}}+1)\norm{\widetilde u(t)}_{H^{s+1}} \right).
\end{align*}
\end{theorem}
\begin{proof}
The proof is a direct application of Theorem 1 of \cite{Kato} which gives the well-posedness result and the estimates: We can rewrite \eqref{linear_euler_equation_V_bar} as 
\begin{equation*}
\partial_t \widetilde u + \sum_{i=1}^2 A_i \partial_i \widetilde u + A_3 \widetilde u = f
\end{equation*}

where $A_i := (u_i + \overline{V}_i)I_2$ for $i \in \{1,2\}$ and $A_3 := \begin{pmatrix} \partial_1 \overline{V}_1 & \partial_2 \overline{V}_1 + 1 \\ \partial_1 \overline{V}_2 - 1 & \partial_2 \overline{V}_2 \end{pmatrix}$.

To apply the theorem we need to prove the following:
\begin{enumerate}

\item $A_i \in C_T  L^2_{\rm{ul}}$ for $1 \leq i \leq 3$
\item $\forall t \in [0,T] \norm{A_i(t)}_{s+1,\rm{ul}} \leq K$ for $1 \leq i \leq 3$
\item $A_1$ and $A_2$ symmetric
\item $\widetilde f \in L^1_T H^{s+1} \cap C_T H^s$
\item $\widetilde u_0 \in H^{s+1}$

\end{enumerate}
where $K := \norm{\overline{V}}_ {L^\infty_T H^{s+2}_{\rm{ul}}} + \norm{u}_{L^\infty_T H^{s+1}_{\rm{ul}}} + C$. The three last points are automatically checked by the assumptions of the theorem. For the first point and the second point, since $u$ is in $\mathcal{C}_T H^{s+1}$, we only need to prove that $\overline{V}$ is in $H^{s+2}_{\rm{ul}}$, which is given by Corollary \ref{corollaire_controle_f_V_bar}.

\end{proof}

As in \cite{Kato} and \cite{Makino} we will use the previous estimates to apply a fixed point theorem $u \mapsto \widetilde u$ on Equation \eqref{linear_euler_equation_V_bar} to prove the well-posedness of the non-linear equation \eqref{euler_equation_V_bar}. Let us first recall that we have fixed $u_0 \in H^{s+1}$, $(\omega,\rho) \in X_T$ (where $X_T$ is defined by \eqref{definition_X_T}) and
\begin{align*}
R_0 &:= R[\rho_0 + \omega_0] \\
M_0 &:= \max(\norm{\rho_0}_{H^s},\norm{\omega_0}_{H^s},\norm{u_0}_{H^{s+1}}) \\
V &:= -\nabla^\bot g\ast(\rho+\omega) \\
\overline{V} &:= -\left(\int\omega_0 + \rho_0\right)\nabla^\bot g\ast\chi\\
f &:= (\overline{V}-V)^\bot - \overline{V}\cdot\nabla\overline{V}.
\end{align*}
Then the well-posedness of \eqref{euler_equation_V_bar} is given by the following theorem:
\begin{theorem}\label{well_posedness_presureless_euler}
Let $s$ be an integer such that $s \geq 3$, then
\begin{enumerate}

\item There exists $T^\ast = T^\ast(M_0,R_0) \leq T$ such that if $T_1 \leq T^\ast$, there is a unique solution $u \in C_{T_1} H^{s+1} \cap C_{T_1}^1 H^s$ to \eqref{euler_equation_V_bar}, with
\begin{equation*}
\norm{u}_{L^\infty_{T_1} H^{s+1}} \leq 2M_0.
\end{equation*}

\item Let $u$ and $u'$ be two solutions defined on $[0,T_1]$ with initial condition $u_0$ and forcing terms $f$ and $f'$, where 
\begin{align*}
f' &:= (\overline{V}-V')^\bot - \overline{V}\cdot\nabla\overline{V} \\
V' &:= -\nabla^\bot g\ast(\rho'+\omega')
\end{align*}
and $(\rho',\omega') \in X_T$. Then we have
\begin{equation*}
\norm{u - u'}_{L^\infty_{T_1} H^{r}} \leq Ce^{C(M_0,R_0)T_1}\norm{V - V'}_{L^1_{T_1} H^{r}}
\end{equation*}
where $0 \leq r \leq s$.
\end{enumerate}
\end{theorem}

\begin{proof}
Let $T_1 \leq T$. We will use a fixed-point method on the following subset of $\mathcal{C}_{T_1}L^2$:
\begin{multline*}
\widetilde{X}_{T_1} := \bigg\{u \in  L^\infty_{T_1} H^{s+1} \cap C_{T_1} H^s \bigg| \norm{u}_{L^\infty_{T_1} H^{s+1}} \leq 2M_0, u(0) = u_0, \\ \norm{u(t) - u(t')}_ {H^s} \leq \widetilde L|t-t'| \; \; \forall t,t' \in [0,T_1] \bigg\}
\end{multline*}
where $\widetilde L$ depends only on $M_0$ and $R_0$ and $c$ are constants to be fixed later. Let $u \in \widetilde X_{T_1}$ and $\widetilde u$ be the solution of \eqref{linear_euler_equation_V_bar} associated to $u$. By Theorem \ref{well_posedness_linear_presureless_euler}, for $t \leq T_1$, we have:
\begin{align*}
\norm{\widetilde u(t)}_{H^{s+1}} &\leq e^{cT_1(\norm{\overline{V}}_ {L^\infty_T H^{s+2}_{\rm{ul}}} + \norm{u}_{L^\infty_T H^{s+1}}+1)} \left(\norm{u_0}_{H^s} + c\norm{f}_{L^1_{T_1} H^{s+1}}\right) \\
&\leq e^{cT_1(C(R_0,M_0) + 2M_0+1)}(M_0 + cT_1C(M_0,R_0))
\end{align*}
by Corollary \ref{corollaire_controle_f_V_bar}. Thus we get
\begin{equation*}
\norm{\widetilde u(t)}_{H^{s+1}} \leq 2M_0 
\end{equation*}
if $T_1$ is small enough. Moreover, using Corollary \ref{corollaire_controle_f_V_bar} again, we get
\begin{align*}
\norm{\partial_t\widetilde u(t)}_{H^s} &\leq c\left(\norm{f(t)}_{H^s} + (\norm{u}_{L^\infty_T H^{s+1}}+\norm{\overline{V}}_ {L^\infty_T H^{s+2}_{\rm{ul}}}+1)\norm{\widetilde u}_{H^{s+1}} \right)  \\
&\leq c(C(M_0,R_0) + (2M_0 + C(M_0,R_0)+1)2M_0) =: \widetilde L
\end{align*}

Thus for all $T_1 \leq T^\ast$ we have built a map $\Psi : \widetilde X_{T_1} \longrightarrow \widetilde X_{T_1}$ such that $\Psi(u) = \widetilde u$, where $T^\ast = T^\ast(M_0,R_0)$. We will now show that $\Psi$ is a contraction for the induced distance on $X_{T_1}$. Let $u$ and $w$ be two elements of $X_{T_1}$ and set $U := u - w$. Then $\widetilde U := \widetilde u - \widetilde w$ satisfies:
\begin{equation*}
\partial_t \widetilde U + ((u + \overline{V})\cdot\nabla) \widetilde U + (\widetilde U \cdot \nabla) \overline{V} = -(U\cdot \nabla) \widetilde w + \widetilde U^\bot.
\end{equation*}
Thus since $(U\cdot\nabla)\tilde w \in \mathcal{C}_T L^2 \cap L^1_T H^1$  we can apply Theorem 1 from \cite{Kato} to have the following estimate:
\begin{align*}
\norm{\widetilde U}_{L^\infty_{T_1} L^2} &\leq e^{cT_1(\norm{\overline{V}}_ {L^\infty_T H^{s+2}_{\rm{ul}}} + \norm{u}_{L^\infty_T H^{s+1}}+1)} \left(0+ c\norm{(U\cdot \nabla)\widetilde w}_{L^1_{T_1} L^2}\right) \\
&\leq e^{cT_1(C(M_0,R_0) + 2M_ 0+1)}cT_1\norm{\nabla \widetilde w}_{L^\infty_{T_1} L^\infty}\norm{U}_{L^\infty_{T_1} L^2} \\
&\leq 4cM_0T_1e^{cT_1(C(M_0,R_0) + 2M_ 0+1)}\norm{U}_{L^\infty_{T_1} L^2}
\end{align*}
using \ref{corollaire_controle_f_V_bar} in the last inequality. Thus $\Psi$ is a contraction if $T$ is small enough, so since $\widetilde X_{T_1}$ is complete (this can be proved in the same way as the closedness of $X_T$ which is proved in the beginning of section \ref{subsection_EEP}), it has a unique fixed point in $\widetilde X_{T_1}$, thus \eqref{euler_equation_V_bar} has a unique solution for short time. Remark that the solution we find belongs to the space $L^\infty_{T_1} H^{s+1} \cap W^{1,\infty}_{T_1} H^s$. Let us justify that it also belongs to $\mathcal{C}_{T_1} H^{s+1} \cap \mathcal{C}^1_{T_1} H^s$:

Let $\epsilon > 0$, $t_1, t_2 \in [0,T_1]$ and $\chi_n$ be a mollifier. We have:
\begin{align*}
\norm{u(t_1) - u(t_2)}_{H^{s+1}} 
\leq& \norm{\chi_n \ast(u(t_1) - u(t_2))}_{H^{s+1}} \\
&+ \norm{[I_2-\chi_n\ast](u(t_1) - u(t_2))}_{H^{s+1}} \\
\leq& C_n\norm{u(t_1) - u(t_2))}_{L^2} + \epsilon
\end{align*}
if $n$ is big enough (see for example Theorem 4.22 of \cite{Brezis}). Thus since $u \in \mathcal{C}_{T_1} L^2$, if $|t_1 - t_2|$ is small enough,
\begin{equation*}
\norm{u(t_1) - u(t_2)}_{H^{s+1}} \leq 2\epsilon
\end{equation*}
Thus $u \in \mathcal{C}_{T_1} H^{s+1}$.  Moreover we have
\begin{equation*}
\partial_t u = - ((u+\overline{V})\cdot \nabla) u - (u \cdot \nabla) \overline{V} + u^\bot + f
\end{equation*}
By assumption $f \in \mathcal{C}_{T_1} H^s$ and by the previous fixed point $u^\bot \in \mathcal{C}_{T_1} H^{s}$. Now using Claim $(1)$ of Proposition \ref{proposition_biot_savart}, $\overline{V} \in \mathcal{C}_{T_1} H^{s+1}_{\rm{ul}}$, so since $s \geq 2$, we have 
\begin{align*}
((u+\overline{V})\cdot \nabla) u &\in \mathcal{C}_{T_1} H^s \\
(u \cdot \nabla) \overline{V} &\in \mathcal{C}_{T_1} H^s
\end{align*}
applying Lemma 2.9 of \cite{Kato} which gives a sufficient condition to have the product of an $H^{s_1}_{\rm{ul}}$ and $H^{s_2}$ function in $H^r$. Thus $u \in \mathcal{C}^1_{T_1} H^s$.

Now let us prove the second point of our theorem: Let $u$ and $u'$ be two solutions associated to $f_1$ and $f_2$ defined on $[0,T_1]$ with $T_1 \leq T^\ast(M_0,R_0)$. Then $U := u - u'$ verifies:
\begin{equation*}
\partial_t U + ((u + \overline{V})\cdot\nabla) U + (U \cdot\nabla)(\overline{V} + u') = U^\bot + F
\end{equation*}
where $F := f - f'$. We can rewrite this equation as
\begin{equation*}
\partial_t U + \sum_{i=1}^2 A_i \partial_i U + B U = F
\end{equation*}
where $A_i := (u_i + \overline{V}_i)I_2$ and $B :=\begin{pmatrix} \partial_1 \overline{V}_1 + \partial_1 u'_1 & 1 + \partial_2 \overline{V}_1 + \partial_2 u'_1 \\ \partial_1 \overline{V}_2 + \partial_1 u'_2 - 1 & \partial_2 \overline{V}_2 + \partial_2 u'_2 \end{pmatrix}$.
Then by Theorem 1 of \cite{Kato}, for any $0 \leq r \leq s$ we have:
\begin{align*}
\norm{U}_{L^\infty_{T_1} H^{r}} &\leq C e^{cT_1(\norm{A_1}_{L^\infty_{T_1} H^s_{\rm{ul}}} + \norm{A_2}_{L^\infty_{T_1} H^s_{\rm{ul}}} + \norm{B}_{L^\infty_{T_1} H^s_{\rm{ul}}})}\norm{F}_{L^1_{T_1} H^{r}} \\
&\leq Ce^{c(\norm{\overline{V}}_ {L^\infty_T H^{s+2}_{\rm{ul}}}+M_0+1)T_1}\norm{F}_{L^1_{T_1} H^{r}} \\
&\leq Ce^{c(C(M_0,R_0)+M_0+1)T_1}\norm{V-V'}_{L^1_{T_1} H^{r}}
\end{align*}
where we used Corollary \ref{corollaire_controle_f_V_bar} in the last inequality.
\end{proof}

\subsection{Continuity equations} \label{subsection_continuity_equations}

In this subsection we still fix $s \geq 3$, $u \in C_T H^{s+1} \cap C_T^1 H^s$, $(\rho,\omega) \in (C_T H^s)^2$, $V := -\nabla^\bot g \ast(\rho + \omega)$, $\chi$ smooth with compact support such that $\int \chi = 1$, $\overline{V} := -\left(\int\omega_0 + \rho_0\right)\nabla^\bot g\ast\chi$, $v := u + \overline{V}$ and we consider the following continuity equations:
\begin{equation}
\label{transport_lineaire}
\left\{
\begin{aligned}
& \partial_t \widetilde \omega + \div(\widetilde \omega V) = 0 \\
& \partial_t \widetilde \rho + \div(\widetilde \rho v) = 0
\end{aligned}\right.
\end{equation} 
with initial conditions $(\rho_0,\omega_0)$.

\begin{theorem}\label{well_posedness_linear_continuity}
Let $u$,$\rho$,$\omega$ be as in the upper paragraph, there exists a solution $(\widetilde \rho, \widetilde \omega) \in \mathcal{C}_T H^s \cap C^1_T H^{s-1}$ of \eqref{transport_lineaire}, unique in $C_T L^2$. Moreover, we have the following estimates:
\begin{align*}
\norm{\widetilde \rho}_{L^\infty_T H^s} &\leq \norm{\rho_0}_{H^s}e^{cT\norm{u}_{L^\infty_T H^s}}\exp\left(ce^{cT\norm{u}_{L^\infty_T H^s}}T\norm{\nabla v}_{L^\infty_T H^s}\right) \\
\norm{\widetilde \omega}_{L^\infty_T H^s} &\leq \norm{\omega_0}e^{cT\norm{\nabla V}_{L^\infty_T H^s}} \\
\norm{\partial_t \widetilde \omega}_{L^\infty_T H^{s-1}} &\leq C(1+R_T[\rho + \omega]^\frac{1}{3})\norm{\rho + \omega}_{L^\infty_TH^s}\norm{\widetilde \omega}_{L^\infty_T H^s} \\
\norm{\partial_t \widetilde \rho}_{L^\infty_T H^{s-1}} &\leq C\left(\left|\int(\rho_0+\omega_0)\right|+\norm{u}_{L^\infty_TH^{s+1}}\right)\norm{\widetilde \rho}_{L^\infty_T H^s}.
\end{align*}

Now let $\widetilde \rho_1$ and $\widetilde \rho_2$ be two solutions associated to two velocity fields $v_1 = u_1 + \overline{V}$ and $v_2 = u_2 + \overline{V}$ with same initial conditions, and $\widetilde \omega_1$ and $\widetilde \omega_2$ be two solutions associated to two velocity fields $V_1$ and $V_2$ with same initial conditions, then we have the following estimates:
\begin{equation*}
\norm{\widetilde \omega_1 - \widetilde \omega_2}_{L^\infty_TL^2} 
\leq  CT\norm{V_1 - V_2}_{L^\infty_T L^2}\norm{\widetilde \omega_2}_{L^\infty_TH^3}
\end{equation*}
\begin{align*}
\norm{\widetilde\rho_1 - \widetilde\rho_2}_{L^\infty_TL^2} \leq CT \norm{\widetilde \rho_2}_{L^\infty_TH^3}\norm{v_2 - v_1}_{L^\infty_TH^1}e^{cT\norm{u_1}_{L^\infty_TH^3}}.
\end{align*}
\end{theorem}

We will also give a general lemma to control the support of a compactly supported solution of a continuity equation:

\begin{lemma}\label{lemme_controle_support}
If $\mu$ is the solution of the following continuity equation,
\begin{equation*}
\partial_t \mu + \div(\mu a) = 0
\end{equation*}
with $a \in C_T W^{1,\infty}$ and $\mu_0$ with compact support, then $\mu$ has compact support and 
\begin{equation}
\label{controle_support}
R_T[\mu] \leq R[\mu_0] + T\norm{a}_{L^\infty_T L^\infty}.
\end{equation}

\end{lemma}

In order to prove the main theorem we will need the following result:
\begin{lemma}\label{lemma_wp_continuity}
If $a$ is a Lipschitz vector field, $\mu_0 \in L^2$ and $f \in L^1_T L^2$ then there exists a unique solution of the continuity equation
\begin{equation*}
\partial_t \mu + \div(\mu a) = f
\end{equation*} 
in $\mathcal{C}_TL^2$. Moreover we have the following estimate
\begin{equation}
\label{estimee_continuite_l2}
\norm{\mu(t)}_{L^2} \leq \left(\norm{\mu_0}_{L^2} + \int_0^t \norm{f(\tau)}_{L^2}\dd\tau\right)e^{c\int_0^t\norm{\div(a)(\tau)}_{L^\infty}\dd\tau}.
\end{equation}
\end{lemma}

\begin{proof}[Proof of Lemma \ref{lemma_wp_continuity}]
The existence and uniqueness of the solution in $\mathcal{C}_TL^2$ can be obtained by Theorem 3.19 and Remark 3.20 of \cite{BahouriCheminDanchin}. Moreover by Proposition 6 of \cite{AmbrosioCrippa}, we know that for all $t$ and almost every $x$ we have
\begin{align*}
\mu(t,X(t,x)) =& \mu_0(x) \\ &+ \int_0^t \bigg(\div(a)(s,X(s,x))\mu(s,X(s,x)) + f(s,X(s,x))\bigg)\dd s
\end{align*} 
where $X$ is the flow associated to $a$. Let us denote $\overline{h}(t,x) = h(t,X(t,x))$ for any function $h$. Taking the $L^2$ norm of the upper inequality we get
\begin{equation*}
\norm{\overline{\mu}(t)}_{L^2} \leq \norm{\mu_0}_{L^2} + \norm{\overline{f}}_{L^1_TL^2} + \int_0^t \norm{\div(a)(s)}_{L^\infty}\norm{\overline{\mu}(s)}_{L^2}. 
\end{equation*} 
Thus by Gronwall lemma,
\begin{equation*}
\norm{\overline{\mu}(t)}_{L^2} \leq (\norm{\mu_0}_{L^2} + \norm{\overline{f}}_{L^1_TL^2})e^{\norm{\div(a)}_{L^1_TL^\infty}}.
\end{equation*}
Now remark that for any $L^2$ function $g$,
\begin{align*}
\int |g(X(t,x))|^2\dd x &= \int |JX^t(x)||g(x)|^2\dd x
&\leq \norm{g}_{L^2}^2e^{\norm{\div(a)}_{L^1_TL^\infty}}
\end{align*}
by inequality (7) of \cite{AmbrosioCrippa}. Using it for $\overline{\mu}$ and $\overline{f}$ we get inequality \eqref{estimee_continuite_l2}.
\end{proof}

Now we prove the main theorem of the section:
\begin{proof}[Proof of Theorem \ref{well_posedness_linear_continuity}]
Let us know use the previous lemma to prove the $H^s$ bound on $\widetilde \omega$. Let $\alpha$ be a multi-index such that $|\alpha| \leq s$. Then, since $V$ is divergent-free,
\begin{equation*}
\partial_t \partial^\alpha \widetilde \omega + \div(V\partial^\alpha \widetilde \omega) = F^\alpha 
\end{equation*}
where $F^\alpha$ is a combination of $\partial^{\alpha_1} V\cdot \partial^{\alpha_2} \nabla \omega$ with $|\alpha_1| + |\alpha_2| = s$, $|\alpha_2| \leq s-1$ and $|\alpha_1| \geq 1$. Thus by the upper estimate \eqref{estimee_continuite_l2}, since $V$ is divergent-free, we have:
\begin{equation*}
\norm{\partial^\alpha \widetilde \omega(t)}_{L^2} \leq \left(\norm{\partial^\alpha\widetilde\omega_0}_{L^2} + \int_0^t \norm{F^\alpha(\tau)}_{L^2}\dd \tau\right).
\end{equation*}
If $|\alpha_1| \leq s-1$, then
\begin{align*}
\norm{\partial^{\alpha_1} V\cdot \partial^{\alpha_2} \nabla \widetilde \omega}_{L^2} 
&\leq \norm{\partial^{\alpha_1 }V}_{L^\infty}\norm{\partial^{\alpha_2} \nabla \widetilde\omega}_{L^2} \\
&\leq \norm{\partial^{\alpha_1} V}_{H^2}\norm{\partial^{\alpha_2} \nabla \widetilde\omega}_{L^2}\\
&\leq \norm{\nabla V}_{H^s}\norm{\widetilde \omega}_{H^s}.
\end{align*}
If $|\alpha_1| = s$, then $\alpha_2 = 0$, thus
\begin{align*}
\norm{\partial^{\alpha_1} V\cdot \partial^{\alpha_2} \nabla \widetilde \omega}_{L^2} 
&\leq \norm{\nabla V}_{H^s}\norm{\nabla \widetilde \omega}_{L^\infty} \\
&\leq \norm{\nabla V}_{H^s}\norm{\nabla \widetilde \omega}_{H^2} \\
&\leq \norm{\nabla V}_{H^s}\norm{\widetilde \omega}_{H^s}.
\end{align*}
Thus
\begin{equation*}
\norm{\partial^\alpha \widetilde \omega(t)}_{L^2} \leq \left(\norm{\partial^\alpha\widetilde\omega_0}_{L^2} + c\int_0^t \norm{\nabla V(\tau)}_{H^s}\norm{\widetilde \omega(\tau)}_{H^s}\dd\tau\right).
\end{equation*}
Summing over all indices $\alpha$, we get
\begin{equation*}
\norm{\widetilde \omega(t)}_{H^s} \leq \left(\norm{\widetilde \omega_0}_{H^s} + c\int_0^t \norm{\nabla V(\tau)}_{H^s}\norm{\widetilde \omega(\tau)}_{H^s}\dd\tau\right).
\end{equation*}

By Grönwall's lemma we get the first inequality of our theorem. Now we will prove the estimate on $\widetilde \rho$. For a multi-index $\alpha$ with $|\alpha| \leq s$ we also have
\begin{equation*}
\partial_t \partial^\alpha \widetilde \rho + \div(v\partial^\alpha \widetilde \rho) = F^\alpha.
\end{equation*}

Because $v$ is not divergent-free, $F^\alpha$ is now a combination of $\partial ^{\alpha_1} v \partial^{\alpha_2}\widetilde \rho$ where $|\alpha_1| +  |\alpha_2| = s+1$, $|\alpha_1| \geq 1$ and $|\alpha_2| \leq s$. If $|\alpha_1| \leq s-1$, we have
\begin{align*}
\norm{\partial^{\alpha_1} v\cdot \partial^{\alpha_2} \widetilde \rho}_{L^2} 
&\leq \norm{\partial^{\alpha_1 }v}_{L^\infty}\norm{\partial^{\alpha_2} \widetilde\rho}_{L^2} \\
&\leq \norm{\partial^{\alpha_1} v}_{H^2}\norm{\partial^{\alpha_2}\widetilde\rho}_{L^2}\\
&\leq \norm{\nabla v}_{H^s}\norm{\widetilde \rho}_{H^s}.
\end{align*}
Now if $|\alpha_1| = s$ or $s+1$ (respectively  $|\alpha_2| = 0$ or $1$),
\begin{align*}
\norm{\partial^{\alpha_1} v\cdot \partial^{\alpha_2} \widetilde \rho}_{L^2} 
&\leq \norm{\nabla v}_{H^s}\norm{\partial^{\alpha_2}\widetilde \rho}_{L^\infty} \\
&\leq \norm{\nabla v}_{H^s}\norm{\partial^{\alpha_2}\widetilde \rho}_{H^2} \\
&\leq \norm{\nabla v}_{H^s}\norm{\widetilde \rho}_{H^s}.
\end{align*}

Thus
\begin{align*}
\norm{\partial^\alpha \widetilde \rho(t)}_{L^2} \leq& \bigg(\norm{\partial^\alpha\widetilde\rho_0}_{L^2}  \\
&+ c\int_0^t \norm{\nabla v(\tau)}_{H^s}\norm{\widetilde \rho(\tau)}_{H^s}\dd\tau\bigg)e^{c\int_0^T\norm{\div(v)}_{L^\infty}(\tau)\dd\tau}.
\end{align*}
Summing over all indices $\alpha$, we get
\begin{align*}
\norm{\widetilde \rho(t)}_{H^s} 
&\leq \left(\norm{\widetilde\rho_0}_{H^s} + c\int_0^t \norm{\nabla v(\tau)}_{H^s}\norm{\widetilde \rho(\tau)}_{H^s}\dd\tau\right)e^{c\int_0^T\norm{\div(v)}_{L^\infty}(\tau)\dd\tau} \\
&\leq \left(\norm{\widetilde\rho_0}_{H^s} + c\int_0^t \norm{\nabla v(\tau)}_{H^s}\norm{\widetilde \rho(\tau)}_{H^s}\dd\tau\right)e^{c\int_0^T\norm{u(\tau)}_{H^s}\dd\tau}
\end{align*}
because $\div(v) = \div(u)$. The corresponding estimate follows by Grönwall's lemma. 

Now let us bound the time derivatives of $\widetilde \omega$ and $\widetilde \rho$. Take $\alpha$ a multi-index with $|\alpha| \leq s-1$, then
\begin{equation*}
\partial_t \partial^\alpha \widetilde \omega = - \partial^\alpha(V\cdot \nabla \widetilde \omega) = \sum_{\alpha_1 + \alpha_2 = \alpha} c_{\alpha_1,\alpha_2} \partial^{\alpha_1} V \cdot \nabla \partial^{\alpha_2}\widetilde \omega.
\end{equation*}
Moreover,
\begin{align*}
\norm{\partial^{\alpha_1} V \cdot \nabla \partial^{\alpha_2}\widetilde \omega}_{L^2} 
&\leq \norm{\partial^{\alpha_1}V}_{L^\infty}\norm{\nabla\partial^{\alpha_2}\widetilde \omega}_{L^2} \\
&\leq C(\norm{V}_{L^\infty} + \norm{\nabla V}_{H^s})\norm{\widetilde \omega}_{H^s}.
\end{align*}
Now by Claim $(3)$ of Proposition \ref{proposition_biot_savart},
\begin{align*}
\norm{V}_{L^\infty}  \leq CR_T[\rho+\omega]^\frac{1}{3}\norm{\rho+\omega}_{H^1}
\end{align*}
and $\norm{\nabla V}_{H^s} \leq C\norm{\rho+\omega}_{H^s}$. Thus we have our estimate.

Let us do the same kind of computations for $\widetilde \rho$: 
\begin{align*}
\partial_t \partial^\alpha \widetilde \rho &= \partial^\alpha (\div(u)\widetilde \rho + u \cdot \nabla \widetilde \rho + \overline{V}\cdot \nabla \widetilde\rho).
\end{align*}
If $|\alpha_1 +  \alpha_2| = s-1$,
\begin{align*}
\norm{\partial^{\alpha_1}u \cdot \partial^{\alpha_2}\nabla \widetilde\rho}_{L^2} 
&\leq \norm{\partial^{\alpha_1}u}_{L^\infty}\norm{\partial^{\alpha_2}\widetilde \nabla \rho}_{L^2} \\
&\leq \norm{u}_{H^{s+1}}\norm{\widetilde \rho}_{H^s}.
\end{align*}
We do the same estimates for every term composing $\partial^\alpha(\div(u)\widetilde \rho)$, except for
\begin{align*}
\norm{\partial^\alpha\div(u)\widetilde \rho}_{L^2} &\leq \norm{u}_{H^s}\norm{\widetilde \rho}_{L^\infty} \\
&\leq \norm{u}_{H^{s+1}}\norm{\widetilde \rho}_{H^s}.
\end{align*}
Now for the third term, if $|\alpha_1 + \alpha_2| = s-1$,
\begin{align*}
\norm{\partial^{\alpha_1}\overline{V}\cdot \nabla \partial^{\alpha_2}\widetilde \rho}_{L^2} &\leq \norm{\partial^{\alpha_1}\overline{V}}_{L^\infty}\norm{\nabla \partial^{\alpha_2}\widetilde \rho}_{L^2} \\
&\leq C\left|\int(\rho_0+\omega_0)\right|\norm{\nabla g\ast\partial^{\alpha_1}\chi}_{L^\infty}\norm{\widetilde \rho}_{H^s} \\
&\leq C\left|\int(\rho_0+\omega_0)\right|\norm{\widetilde \rho}_{H^s}
\end{align*}
by Claim $(3)$ of Proposition \ref{proposition_biot_savart}. Thus we have the estimate we wanted to prove.

Now let us prove the last point of our theorem. Substracting the two continuity equations satisfied by $\widetilde \omega_1$ and $\widetilde \omega_2$, we have
\begin{equation*}
\partial_t (\widetilde \omega_1 - \widetilde \omega_2) + \div(V_1(\widetilde \omega_1 - \widetilde \omega_2)) = (V_2 - V_1)\cdot \nabla \widetilde \omega_2.
\end{equation*}
Using estimate \eqref{estimee_continuite_l2}, we have
\begin{align*}
\norm{\widetilde \omega_1 - \widetilde \omega_2}_{L^\infty_TL^2} 
&\leq c\int_0^T \norm{(V_1 - V_2)\cdot \nabla \widetilde \omega_2}_{L^2}(\tau)\dd\tau \\
&\leq CT\norm{V_1 - V_2}_{L^\infty_T L^2}\norm{\widetilde \omega_2}_{H^3}.
\end{align*}

Now we prove the last estimate we need for $\widetilde \rho_1 - \widetilde \rho_2$:
\begin{equation*}
\partial_t (\widetilde \rho_1 - \widetilde \rho_2) + \div(v_1(\widetilde \rho_1 - \widetilde \rho_2)) = (v_2 - v_1)\cdot \nabla \widetilde \rho_2 + \div(v_2 - v_1)\widetilde \rho_2.
\end{equation*}
We can bound the second term the same way that we did for the previous one:
\begin{align*}
\norm{(v_2 - v_1)\cdot \nabla \widetilde \rho_2 + \div(v_2 - v_1)\widetilde \rho_2}_{L^2} 
\leq& \norm{v_1 - v_2}_{L^2}\norm{\nabla \widetilde \rho_2}_{L^\infty} \\
&+ \norm{\div(v_1 - v_2)}_{L^2}\norm{\widetilde \rho_2}_{L^\infty} \\
&\leq 2\norm{v_1 - v_2}_{H^1}\norm{\widetilde \rho_2}_{H^3}.
\end{align*}
Thus by \eqref{estimee_continuite_l2},
\begin{align*}
\norm{\widetilde \rho_1 - \widetilde \rho_2}_{L^\infty_TL^2} 
&\leq CT \norm{\widetilde \rho_2}_{L^\infty_TH^3}\norm{v_2 - v_1}_{L^\infty_TH^1}e^{cT\norm{\div(v_1)}_{L^\infty_TH^2}} \\
&\leq CT \norm{\widetilde \rho_2}_{L^\infty_TH^3}\norm{v_2 - v_1}_{L^\infty_TH^1}e^{cT\norm{u_1}_{L^\infty_TH^3}}
\end{align*}
because $\div(v_1) = \div(u_1)$.
\end{proof}

Now let us prove lemma \ref{lemme_controle_support}:
\begin{proof}[Proof of Lemma \ref{lemme_controle_support}]
Solving the continuity equation by characteristics, we see that
\begin{equation*}
\supp(\mu(t)) = \psi^t(\supp(\mu(0)))
\end{equation*}
where $\psi$ is the flow associated to $a$. Moreover, for $x \in \supp{\mu_0}$,
\begin{align*}
|\psi^t(x)| 
&\leq |\psi^0(x)| + |\psi^t(x)-\psi^0(x)| \\
&\leq |x| + \left|\int_0^t a(\tau,\psi^\tau(x))\dd\tau\right| \\
&\leq R[\mu_0] + T\norm{a}_{L^\infty_T L^\infty}.
\end{align*}
Taking the supremum for all $x$ in $\supp(\mu_0)$, we get \eqref{controle_support}.

\end{proof}

\subsection{Monokinetic spray System}\label{subsection_EEP}

In this section we prove the well-posedness result of system \eqref{edp}, that is Theorem \ref{main_theorem_well_posedness}: 
\begin{proof}[Proof of Theorem \ref{main_theorem_well_posedness}]
Let $(\rho_0,\omega_0) \in H^s$, $u_0 \in H^{s+1}$ and $\chi$ be a smooth function with compact support such that $\int \chi = 1$. We recall that we have defined
\begin{equation*}
M_0 := \max(\norm{\rho_0}_{H^s},\norm{\omega_0}_{H^s},\norm{u_0}_{H^{s+1}}), \; R_0 := R[\rho_0 + \omega_0]
\end{equation*}
and
\begin{align*}
X_T := \; \bigg\{ &(\omega,\rho) \in  L^\infty_T H^s \cap C_T H^{s-1} \bigg | \omega(0) = \omega_0, \rho(0) = \rho_0, \\
&\norm{\rho}_{L^\infty_T H^s} \leq 2M_0, \norm{\omega}_{L^\infty_T H^s} \leq 2M_0, R_T[\rho + \omega] \leq 2 R_0, \\
&\forall t \in [0,T], \int(\rho(t)+\omega(t)) = \int (\rho_0+\omega_0), \\
&\forall t,t' \in [0,T], \norm{\rho(t) - \rho(t')}_{H^{s-1}} \leq L|t-t'|, \\ 
&\norm{\omega(t) - \omega(t')}_{H^{s-1}} \leq L|t-t'|, \bigg\}
\end{align*}
with $L > 0$ that we will fix later. Let us justify that $X_T$ is the complete metric space for the distance
\begin{align*}
d((\rho_1,\omega_1),(\rho_2,\omega_2)) := \norm{\rho_1 - \rho_2}_{L^\infty_TL^2} + \norm{\omega_1 - \omega_2}_{L^\infty_TL^2}.
\end{align*}
It is sufficient to prove that $X_T$ is closed in $(L_T^\infty L^2)^2$. Let us consider a sequence of functions $(\rho_N,\omega_N)$ in $X_T$ and $(\rho,\omega) \in (L^\infty_T L^2)^2$ such that 
\begin{equation*}
d((\rho_N,\omega_N),(\rho,\omega)) \Tend{N}{\infty} 0
\end{equation*}
and prove that $(\rho,\omega) \in X_T$. By Banach-Alaoglu's theorem, since $H^s$ is a Hilbert space, for almost every time there exists a subsequence $\rho_{\phi_t(n)}(t)$ that converges weakly in $H^s$. Thus by uniqueness of the limit in weak $L^2$  $\rho_{\phi_t(n)}(t)$ converges weakly to $\rho(t)$ for almost every $t \in [0,T]$. By lower semi-continuity of the $H^s$ norm we get that
\begin{equation}\label{X_ferme_estimee_1a}
\norm{\rho}_{L^\infty_T H^s} \leq 2M_0.
\end{equation}
By the same kind of argument we can prove that
\begin{equation}\label{X_ferme_estimee_1b}
\norm{\omega}_{L^\infty_T H^s} \leq 2M_0.
\end{equation}
and that for all $t,t' \in [0,T]$
\begin{equation}\label{X_ferme_estimee_2}
\begin{aligned}
\norm{\rho(t) - \rho(t')}_{H^{s-1}} &\leq L|t-t'| \\ 
\norm{\omega(t) - \omega(t')}_{H^{s-1}} &\leq L|t-t'|.
\end{aligned}
\end{equation}
As a consequence $\rho$ and $\omega$ are continuous in time with value in $H^{s-1}$ and thus 
\begin{equation}\label{X_ferme_estimee_3}
\begin{aligned}
\omega(0) &= \omega_0 \\
\rho(0) &= \rho_0.
\end{aligned}
\end{equation}
Moreover for all $t \in [0,T]$,
\begin{equation*}
\int \mathbf{1}_{B(0,2R_0)}(\rho_N^2(t)+\omega_N^2(t)) \Tend{N}{+\infty} \int \mathbf{1}_{B(0,2R_0)}(\rho^2(t)+\omega^2(t)) = 0
\end{equation*}
by strong convergence in $L^2$. Thus $\rho$ and $\omega$ have compact support and
\begin{equation}\label{X_ferme_estimee_4}
R[\rho + \omega] \leq 2R_0.
\end{equation}
Finally, compact support and convergence in $L^2$ implies convergence in $L^1$ so we get that for every $t \in [0,T]$,
\begin{equation}\label{X_ferme_estimee_5}
\int(\rho(t)+\omega(t)) = \int (\rho_0+\omega_0).
\end{equation}
Inequalities \eqref{X_ferme_estimee_1a}, \eqref{X_ferme_estimee_1b}, \eqref{X_ferme_estimee_2}, \eqref{X_ferme_estimee_3}, \eqref{X_ferme_estimee_4} and \eqref{X_ferme_estimee_5} gives us that $(\rho,\omega) \in X_T$, so $X_T$ is closed in $L^\infty_T L^2$.

Now let us build a contraction $X_T \longrightarrow X_T$. For $(\rho,\omega) \in X_T$ fixed, we have defined
\begin{itemize}
\item $V := -\nabla^\bot g \ast (\rho + \omega)$
\item $\overline{V} := -\left(\int\rho_0+\omega_0\right)\nabla^\bot g \ast \chi$
\item $f := (V - \overline{V})^\bot - (\overline{V}\cdot\nabla) \overline{V}$.
\end{itemize}

By Corollary \ref{corollaire_controle_f_V_bar}, $f \in L^\infty_T H^{s+1}\cap \mathcal{C}_T H^s$. Let $T_1$ be sufficiently small so that Theorem \ref{well_posedness_presureless_euler} can be applied and $u$ be the solution of \eqref{euler_equation_V_bar} given by this theorem, $v = u + \overline{V}$, and $(\widetilde \rho,\widetilde \omega)$ be the solution of \eqref{transport_lineaire} given by Theorem \ref{well_posedness_linear_continuity}. According to Theorem \ref{well_posedness_presureless_euler}, the smallness of $T_1$ depends on $M_0$ and $R_0$. Now let us justify that for small enough $T_2 \leq T_1$, we have $(\widetilde \rho,\widetilde \omega) \in X_{T_2}$. By Theorem \ref{well_posedness_presureless_euler}, we have the following estimates:
\begin{align*}
\norm{\widetilde \rho}_{L^\infty_{T_1} H^s} &\leq \norm{\rho_0}_{H^s}e^{cT_1\norm{u}_{L^\infty_{T_1} H^s}}\exp\left(ce^{cT_1\norm{u}_{L^\infty_{T_1} H^s}}T_1\norm{\nabla v}_{L^\infty_{T_1} H^s}\right) \\
\norm{\widetilde \omega}_{L^\infty_{T_1} H^s} &\leq \norm{\omega_0}e^{cT_1\norm{\nabla V}_{L^\infty_{T_1} H^s}}.
\end{align*}
Remark that
\begin{equation*}
\norm{\nabla V}_{L^\infty_{T_1} H^s} \leq C\norm{\rho + \omega}_{L^\infty_{T} H^s} \leq 4CM_0
\end{equation*}
by Claim $(2)$ of Proposition \ref{proposition_biot_savart}. Moreover by Claim $(2)$ of Proposition \ref{proposition_biot_savart} and Theorem \ref{well_posedness_presureless_euler},
\begin{align*}
\norm{\nabla v}_{L^\infty_{T_1} H^s}
&\leq C(\norm{u}_{L^\infty_{T_1} H^{s+1}} + \norm{\nabla \overline{V}}_{L^\infty_{T_1} H^s}) \\
&\leq C(2M_0 + CR_0M_0) \\
&\leq C(1+R_0)M_0.
\end{align*}
Thus $\norm{\widetilde \rho}_{L^\infty_{T_2} H^s} \leq 2M_0$ and $\norm{\widetilde \omega}_{L^\infty_{T_2} H^s} \leq 2M_0$ if $T_2 \leq T_1$ and $T_2$ small enough with respect to $M_0$ and $R_0$. Now, by Lemma \ref{lemme_controle_support} and Claim $(3)$ of Proposition \ref{proposition_biot_savart}, if $0 \leq t \leq T_2$ we have
\begin{align*}
R[\widetilde \rho(t) + \widetilde \omega(t)] &\leq R[\widetilde \rho(t)] + R[\widetilde \omega(t)] \\
&\leq R_0 + t(\norm{v}_{L^\infty} + \norm{V}_{L^\infty}) \\
&\leq R_0 + t(\norm{u}_{H^s} + \norm{\overline{V}}_{L^\infty} + \norm{V}_{L^\infty}) \\
&\leq R_0 + t(2M_0 +  C\left|\int \rho_0 + \omega_0\right| + CR_{T_2}[\rho+\omega]^\frac{1}{3}\norm{\rho+\omega}_{H^1}) \\
&\leq R_0 + T_2(2M_0 + 2CR_0M_0 + 4CR_0^\frac{1}{3}M_0) \\
&\leq 2R_0
\end{align*}
if $T_2$ is small enough with respect to $R_0$ and $M_0$. By Theorem 
\ref{well_posedness_linear_continuity}, we have: 
\begin{align*}
\norm{\partial_t \widetilde \omega}_{L^\infty_{T_2}H^{s-1}} 
&\leq C(1+R_{T_2}[\rho + \omega]^\frac{1}{3})\norm{\rho + \omega}_{L^\infty_{T_2}H^s}\norm{\widetilde \omega}_{L^\infty_{T_2}H^s} \\
&\leq C(1+(2R_0)^\frac{1}{3})4M_02M_0 \\
\norm{\partial_t \widetilde \rho}_{L^\infty_{T_2}H^{s-1}} 
&\leq C\left(\left|\int \rho_0 + \omega_0\right|+\norm{u}_{L^\infty_{T^2}H^{s+1}}\right)\norm{\widetilde \rho}_{L^\infty_{T_2}H^s} \\
&\leq C\big(2R_0M_0+2M_0\big)2M_0.
\end{align*} 
Choosing $L$ large enough (with respect to $M_0$ and $R_0$), we have 
\begin{align*}
\norm{\partial_t \widetilde \omega}_{L^\infty_{T_2}H^{s-1}} &\leq L \\
\norm{\partial_t \widetilde \rho}_{L^\infty_{T_2}H^{s-1}} &\leq L.
\end{align*} 
Thus we have built a map $\Phi : (\rho,\omega) \mapsto (\widetilde \rho, \widetilde \omega)$ such that $\Phi(X_{T_2}) \subset X_{T_2}$. We will now prove that $\Phi$ is a contraction for the $L^\infty_{T_2}L^2$ norm.

Let $(\rho_1,\omega_1), (\rho_2,\omega_2) \in X_{T_2}$, $(\widetilde \rho_1,\widetilde \omega_1) = \Phi(\rho_1,\omega_1)$ and $(\widetilde \rho_2,\widetilde \omega_2) = \Phi(\rho_2,\omega_2)$. By Theorem \ref{well_posedness_linear_continuity}, we have
\begin{equation*}
\norm{\widetilde \omega_1 - \widetilde \omega_2}_{L^\infty_{T_2}L^2} 
\leq  CT_2\norm{V_1 - V_2}_{L^\infty_{T_2} L^2}\norm{\widetilde \omega_2}_{L^\infty_{T_2}H^3}
\end{equation*}
and
\begin{align*}
\norm{\widetilde\rho_1 - \widetilde\rho_2}_{L^\infty_{T_2} L^2} 
&\leq CT_2 \norm{\widetilde \rho_2}_{L^\infty_{T_2}H^3}\norm{v_2 - v_1}_{L^\infty_{T_2}H^1}e^{cT\norm{u_1}_{L^\infty_{T_2}H^3}}.
\end{align*}
Moreover, by Theorem \ref{well_posedness_presureless_euler}:
\begin{equation*}
\norm{v_2 - v_1}_{L^\infty_{T_2} H^1} = \norm{u_2 - u_1}_{L^\infty_{T_2} H^1} \leq Ce^{C(M_0,R_0)T_2}\norm{V_1 - V_2}_{L^\infty_{T_2} H^1}
\end{equation*}
Thus
\begin{align*}
\norm{\widetilde \omega_1 - \widetilde \omega_2}_{L^\infty_{T_2}L^2} &+ \norm{\widetilde\rho_1 - \widetilde\rho_2}_{L^\infty_{T_2}L^2} \\
\leq& \; 2CT_2M_0 \norm{V_1 - V_2}_{L^\infty_{T_2} L^2} \\
&+ 2CT_2M_0e^{2cT_2M_0}Ce^{C(M_0,R_0)T_2}\norm{V_1 - V_2}_{L^\infty_{T_2} H^1} \\
\leq& \; C(M_0,R_0,T)T_2\norm{V_1 - V_2}_{L^\infty_{T_2} H^1}
\end{align*}
for any $T_2 \leq T_1 \leq T$. Moreover,
\begin{align*}
&\norm{V_1 - V_2}_{L^\infty_{T_2} H^1} \\
&\leq \norm{V_1 - V_2}_{L^\infty_{T_2} L^2} + \norm{\nabla(V_1 - V_2)}_{L^\infty_{T_2} L^2} \\
&\leq C(1+(4R_0 + (4R_0)^3)^\frac{1}{2})\norm{\rho_1 + \omega_1 - \rho_2 - \omega_2}_{L^\infty_{T_2}L^2}
\end{align*}
by Claims $(2)$ and $(6)$ of Proposition \ref{proposition_biot_savart}. Thus $\Phi$ is a contraction if $T_2$ is small enough (with respect to $M_0$ and $R_0$), so it has a unique fixed point $(\rho,\omega) \in X_{T_2}$.
\end{proof}

\section{Mean-field limit}\label{section:3}

In this section we prove Proposition \ref{asmptotic_positivity_energy}, Theorem \ref{theoreme_champ_moyen}, Proposition \ref{proposition_coerciveness_energy} and Proposition \ref{proposition_donnees_initiales_bien_preparees}. Let us begin by proving Proposition \ref{asmptotic_positivity_energy}. 

\subsection{Proof of Proposition \ref{asmptotic_positivity_energy}}

For $0 < \eta < 1$ we define
\begin{equation*}
g^{(\eta)}(x) =
\left\{
\begin{aligned}
&-\frac{1}{2\pi}\ln(\eta) \qquad &\text{if} \; |x| \leq \eta \\
&g(x) \qquad &\text{if} \; |x| \geq \eta
\end{aligned}\right.
\end{equation*}
and we denote $\delta_{y}^{(\eta)}$ the uniform probability measure on the circle of center $y$ and radius $\eta$. We have the following lemma:

\begin{lemma}\label{egalite_convolution_g_eta}
For any $0 < \eta < 1$ and $y \in \mathbb{R}^2$,
\begin{equation*}
\int g(x-z) \dd \delta_y^{(\eta)}(z) = g^{(\eta)}(x-y)
\end{equation*}
\end{lemma}

\begin{proof}
By a change of variable we may assume that $y = 0$. The function
\begin{equation*}
f(x) = \int_{\partial B(0,\eta)} g(x-z) \dd \delta_0^{(\eta)}(z) 
\end{equation*}
is locally bounded and satisfies $\Delta f = -\delta_0^{(\eta)} = \Delta g^{(\eta)}$. Now if $|x| \geq \eta$, we have 

\begin{align*}
\int_{\partial B(0,\eta)} g(x-z)\dd \delta_0^{(\eta)}(z)  - g^{(\eta)}(x) &= \int_{\partial B(0,\eta)}(g(x-z)-g(x))\dd \delta_0^{(\eta)}(z) \\
&= \int_{\partial B(0,\eta)}g\bigg(\frac{x}{|x|}-\frac{z}{|x|}\bigg)\dd \delta_0^{(\eta)}(z) \\
&\Tend{|x|}{\infty} \int_{\partial B(0,\eta)} -\frac
{1}{2\pi}\ln(1) = 0
\end{align*}
by dominated convergence theorem. Therefore $f - g^{(\eta)}$ is a harmonic bounded function so it is constant. Since $f(z) = g(\eta) = g^{(\eta)}(z)$ for any $z$ of norm $\eta$, we get that $f = g^{(\eta)}$. 
\end{proof}

Integrating by parts, since $\displaystyle{\int \omega-\omega_N = 0}$, we have 
\begin{equation}\label{equality_kinetic_energy_mod_energy}
\norm{\nabla g\ast(\omega-\omega_N)}^2_{L^2} = \iint_{\mathbb{R}^2\times\mathbb{R}^2} g(x-y)(\omega-\omega_N)(x)(\omega-\omega_N)(y) \dd x \dd y.
\end{equation}
For a more detailed justification of such integrations by parts we refer to \cite[Equality (1.23)]{Serfaty}. Therefore we only need to justify Inequality \eqref{lower_bound_potential_energy} to get that $\mathcal{H} \geq 0$.  For that purpose we define
\begin{equation*}
\rho_N^{(\eta)} := \frac{1}{N}\sum_{i=1}^N \delta_{q_i}^{(\eta)}. 
\end{equation*}
We have
\begin{align*}
\iint_{(\mathbb{R}^2\times\mathbb{R}^2)\backslash \Delta} &g(x-y)(\rho + \omega - \rho_N - \omega_N)^{\otimes 2}(\dd x\dd y) \\
=& \iint_{(\mathbb{R}^2\times\mathbb{R}^2)} g(x-y)(\rho + \omega - \rho_N^{(\eta)} - \omega_N)^{\otimes 2}(\dd x\dd y) \\
&+ \iint_{(\mathbb{R}^2\times\mathbb{R}^2)\backslash \Delta}g(x-y)(\dd \rho_N(x)\dd\rho_N(y)-\dd \rho_N^{(\eta)}(x)\dd\rho_N^{(\eta)}(y)) \\
&+2\iint_{(\mathbb{R}^2\times\mathbb{R}^2)} g(x-y)(\omega-\omega_N+\rho)(x)\dd x \dd(\rho_N^{(\eta)}-\rho_N)(y) \\
=& L_1 + L_2 + L_3.
\end{align*}
Integrating by parts the first line we find that
\begin{equation*}
L_1 = \int |\nabla g \ast(\rho + \omega - \rho_N^{(\eta)} - \omega_N)|^2 \geq 0. 
\end{equation*}
For the second line, by Lemma \ref{egalite_convolution_g_eta} we have
\begin{equation*}
L_2 = \frac{1}{N^2}\sum_{1 \leq i \neq j \leq N}\int_{\mathbb{R}^2} (g(q_i - q_j) - g^{(\eta)}(q_i - y))\dd(\delta_{q_i} + \delta_{q_i}^{(\eta)})(y)\dd y 
\end{equation*}
This quantity have been bounded in \cite[Inequality 2.14]{NguyenRosenzweigSerfaty} so we get
\begin{equation*}
L_2 \geq -\frac{C}{N}\sum_{i=1}^N \eta_i^2.
\end{equation*}   
Finally,
\begin{align*}
|L_3| \leq C|g\ast(\omega-\omega_N+\rho)|_{C^\gamma}\eta^\gamma
\end{align*}
so by Morrey's inequality (see for example \cite[Theorem 9.12]{Brezis}) and Hardy-Littlewood-Sobolev inequality (see for example \cite[Theorem 1.7]{BahouriCheminDanchin}) for some $p > 2$ we have
\begin{align*}
|L_3| &\leq C\norm{\nabla g\ast(\omega-\omega_N+\rho)}_{L^p}\eta^\gamma \\
&\leq C\norm{\omega-\omega_N+\rho}_{L^\frac{2p}{p+2}} \\
&\leq C(\norm{\omega}_{L^1 \cap L^\infty} + \norm{\omega_N}_{L^1 \cap L^\infty}+\norm{\rho}_{L^1 \cap L^\infty})\eta^\gamma.
\end{align*}
We get Inequality \eqref{lower_bound_potential_energy} by taking $\eta = N^{-1}$.

\subsection{Proof of Theorem \ref{theoreme_champ_moyen}}
We want to compute the derivative of the functional $\mathcal{H}_N = T_1 + T_2 + T_3 + T_4 + T_5 + T_6 + T_7$ defined in \eqref{definition_H_N}. We will denote $\alpha := \omega + \rho$ and $\alpha_N := \omega_N +\rho_N$.
\begin{align*}
T_1 &:= \frac{1}{N}\sum_{i=1}^N |v(q_i) - p_i|^2 \\
T_2 &:=  \int_{\mathbb{R}^2\times\mathbb{R}^2\backslash \Delta} g(x-y)\alpha_N(t,\dd x)\alpha_N(t,\dd y) \\
T_3 &:= \int_{\mathbb{R}^2\times\mathbb{R}^2\backslash \Delta} g(x-y)\alpha(t,x)\alpha(t,y)\dd x\dd y \\
T_4 &:= -2\int_{\mathbb{R}^2\times\mathbb{R}^2\backslash \Delta} g(x-y)\alpha(t,x)\dd x\dd \alpha_N(t,y)\\
T_5 &:= \norm{\omega(t)-\omega_N(t)}^2_{L^2} \\
T_6 &:= \iint_{\mathbb{R}^2\times\mathbb{R}^2}g(x-y)(\omega-\omega_N)(x)(\omega-\omega_N)(y) \dd x \dd y \\
T_7 &= BN^{-\gamma}.
\end{align*}

\begin{claim}\label{derivee_H_termes_1}
For every $t \in [0,T]$, we have
\begin{align*}
\frac{\dd T_1}{\dd t} 
=& -\frac{2}{N}\sum_{i=1}^N\nabla v(q_i):(v(q_i)-p_i)^{\otimes 2} \\
&- \frac{2}{N}\sum_{i=1}^N p_i\cdot\left(\frac{1}{N}\underset{j\neq i}{\sum_{j=1}^N}\nabla g(q_i-q_j) + \nabla g\ast(\omega_N-\alpha)(q_i)\right) \\
&+ 2\iint_{\mathbb{R}^2\times\mathbb{R}^2\backslash \Delta}v(t,x)\cdot \nabla g(x-y) \rho_N(t,\dd x)(\alpha_N - \alpha)(t,\dd y) \\
=:& \; T_{1,1} + T_{1,2} + T_{1,3}.
\end{align*}
\end{claim}
\begin{proof}[Proof of Claim \ref{derivee_H_termes_1}]
Since $v \in \mathcal{C}^1([0,T]\times\mathbb{R}^2,\mathbb{R}^2)$ we can compute
\begin{align*}
\frac{\dd T_1}{\dd t} 
=& \frac{2}{N}\sum_{i=1}^N(v(q_i)-p_i)\cdot(\partial_t v(q_i) + (p_i\cdot \nabla) v(q_i) - \dot p_i) \\
=& \frac{2}{N}\sum_{i=1}^N(v(q_i)-p_i)\cdot \bigg(-(v\cdot \nabla) v(q_i) + (v-V)^\bot(q_i) \\
&+ (p_i\cdot \nabla) v(q_i) - p_i^\bot + \nabla g \ast \omega_N(q_i) +  \frac{1}{N}\underset{j\neq i}{\sum_{j=1}^N}\nabla g(q_i-q_j)\bigg) \\
=& \frac{2}{N}\sum_{i=1}^N(v(q_i)-p_i)\cdot \bigg(((p_i-v(t,q_i))\cdot \nabla) v(q_i) + (v(q_i) - p_i)^\bot \\ &+\nabla g \ast(\omega_N-\alpha)(q_i) +  \frac{1}{N}\underset{j\neq i}{\sum_{j=1}^N}\nabla g(q_i-q_j)\bigg) \\
=& -\frac{2}{N}\sum_{i=1}^N\nabla v(q_i):(v(q_i)-p_i)^{\otimes 2} \\ 
&+ \frac{2}{N}\sum_{i=1}^N(v(q_i)-p_i)\cdot\bigg(\nabla g\ast(\omega_N-\alpha)(q_i)
+ \frac{1}{N}\underset{j\neq i}{\sum_{j=1}^N}\nabla g(q_i-q_j)\bigg) \\ 
=& -\frac{2}{N}\sum_{i=1}^N\nabla v(q_i):(v(q_i)-p_i)^{\otimes 2} \\
&- \frac{2}{N}\sum_{i=1}^N p_i\cdot\bigg(\frac{1}{N}\underset{j\neq i}{\sum_{j=1}^N}\nabla g(q_i-q_j) + \nabla g\ast(\omega_N-\alpha)(q_i)\bigg) \\
&+ \frac{2}{N}\sum_{i=1}^N v(q_i)\cdot\bigg(\frac{1}{N}\underset{j\neq i}{\sum_{j=1}^N}\nabla g(q_i-q_j) + \nabla g\ast(\omega_N-\alpha)(q_i)\bigg) \\
=& -\frac{2}{N}\sum_{i=1}^N\nabla v(q_i):(v(q_i)-p_i)^{\otimes 2} \\
&- \frac{2}{N}\sum_{i=1}^N p_i\cdot\bigg(\frac{1}{N}\underset{j\neq i}{\sum_{j=1}^N}\nabla g(q_i-q_j) + \nabla g\ast(\omega_N-\alpha)(q_i)\bigg) \\
&+ 2\iint_{\mathbb{R}^2\times\mathbb{R}^2\backslash \Delta}v(t,x)\cdot \nabla g(x-y) \rho_N(t,\dd x)\dd(\alpha_N - \alpha)(t,y) \\
=&  T_{1,1} + T_{1,2} + T_{1,3}.
\end{align*}
\end{proof}

In the incoming computations, we will find some terms which look like $T_{1,2}$, that is, terms depending on $p_i$ (which will cancel out) or like $T_{1,3}$, that is terms of the form:
\begin{equation*}
\iint_{\mathbb{R}^2\times\mathbb{R}^2\backslash \Delta}A(t,x)\cdot \nabla g(x-y) \dd \mu(x)\dd\nu(y)
\end{equation*}
with $A$ a smooth vector field (for example $v$ or $V$) and $\mu$, $\nu$ some signed finite measures (for example $\alpha$ or $\rho - \rho_N$). We will finish our computations grouping all terms corresponding to the same vector field $A$. Let us now compute the time derivative of $T_2$. Notice that the energy
\begin{equation*}
\mathcal{E}_N = \frac{1}{N}\sum_{i=1}^N |p_i|^2 + \iint_{\mathbb{R}^2\times\mathbb{R}^2\backslash \Delta} g(x-y)\dd\alpha_N(t,x)\dd\alpha_N(t,y)
\end{equation*}
of System \eqref{edo} is constant in time (for more details see  \cite[Proposition 5.1]{LacaveMiot}). Thus we have
\begin{equation*}
T_2 = \mathcal{E}_N  - \frac{1}{N}\sum_{i=1}^N |p_i|^2
\end{equation*}
and
\begin{equation}
\label{derivee_H_termes_2}
\begin{aligned}
\frac{\dd T_2}{\dd t}
&= - \frac{2}{N}\sum_{i=1}^N p_i \cdot\bigg(p_i^\bot - \nabla g \ast \omega_N(q_i) - \frac{1}{N}\underset{j \neq i}{\sum_{j=1}^N}\nabla g(q_i - q_j)\bigg) \\
&= \frac{2}{N}\sum_{i=1}^N p_i \cdot\bigg(\nabla g \ast \omega_N(q_i) + \frac{1}{N}\underset{j \neq i}{\sum_{j=1}^N}\nabla g(q_i - q_j)\bigg) \\
&=: T_{2,1}.
\end{aligned}
\end{equation}

Let us compute the time derivative of the third term:
\begin{claim}\label{derivee_H_termes_3}
$T_3 \in W^{1,\infty}([0,T])$ and for almost every $t \in [0,T]$, we have
\begin{align*}
\frac{\dd T_3}{\dd t} 
&=2 \iint_{\mathbb{R}^2\times\mathbb{R}^2}v(t,x)\cdot\nabla g(x-y)\rho(t,x)\alpha(t,y)\dd x\dd y. \\
&=: T_{3,1}.
\end{align*}
\end{claim}

\begin{proof}[Proof of Claim \ref{derivee_H_termes_3}]
Let $(g_\eta)_{0 < \eta < 1}$ be a family of smooth functions such that
\begin{itemize}
\item $g_\eta(x) = g(x) \quad \textrm{if} \quad |x| \geq \eta$,
\item $|g_\eta(x)| \leq |g(x)|$,
\item $\displaystyle{|\nabla g_\eta(x)| \leq \frac{C}{|x|}}$.
\item $g_\eta(-x) = g_\eta(x)$.
\end{itemize}
For $0 \leq s,t \leq T$ and $0 < \eta < 1$, we have
\begin{equation*}
T_3(t) -T_3(s) 
= \iint g(x-y)(\alpha(t,x)\alpha(t,y) - \alpha(s,x)\alpha(s,y))\dd x\dd y.
\end{equation*}
Remark that
\begin{align*}
\iint |&g_\eta(x-y)(\alpha(t,x)\alpha(t,y) - \alpha(s,x)\alpha(s,y))|\dd x\dd y \\
&\leq \iint |g(x-y)|(\alpha(t,x)\alpha(t,y) - \alpha(s,x)\alpha(s,y))|\dd x\dd y < +\infty
\end{align*}
because $\alpha$ has compact support. Thus by dominated convergence theorem we have
\begin{equation}\label{limite_T3_1}
T_3(t) -T_3(s) 
= \underset{\eta \rightarrow 0}{\lim}\iint g_\eta(x-y)(\alpha(t,x)\alpha(t,y) - \alpha(s,x)\alpha(s,y))\dd x\dd y.
\end{equation}
Since $g_\eta$ is smooth and $\alpha$ has compact support, we have by \eqref{def_solution_faible_eq_continuite} that
\begin{equation*}
\int g_\eta(x-y)(\alpha(t,x)-\alpha(s,x))\dd x = \int_s^t\int(\rho v + \omega V)(\tau,x)\cdot \nabla g_\eta(x-y)\dd x\dd \tau
\end{equation*}
Since $(\rho v + \omega V)\cdot\ast\nabla g_\eta \in L^\infty([0,T],\mathcal{C}^1(\mathbb{R}^2,\mathbb{R}))$, we get from the upper equation that $g_\eta \ast \alpha \in W^{1,\infty}([0,T],\mathcal{C}^1(\mathbb{R}^2,\mathbb{R}))$ and that for almost every $t \in [0,T]$,
\begin{equation*}
\partial_t(g_\eta\ast\alpha) = -(\rho v+ \omega V)\cdot\ast\nabla g_\eta.
\end{equation*}
Thus we can use $g_\eta\ast\alpha$ as a test function in \eqref{def_solution_faible_eq_continuite} to get
\begin{align*}
\iint &g_\eta(x-y)(\alpha(t,x)\alpha(t,y) - \alpha(s,x)\alpha(s,y))\dd x\dd y \\
&= -\int_s^t \int ((\rho v + \omega V)\cdot\ast\nabla g_\eta)\alpha + \int_s^t \int (\rho v + \omega V)\cdot \nabla (g_\eta\ast\alpha) \\
&= -\int_s^t \iint (\rho v + \omega V)(\tau,x)\cdot \nabla g_\eta(y-x) \alpha(\tau,y)\dd x\dd y\dd \tau \\
&+ \int_s^t \iint (\rho v + \omega V)(\tau,x)\cdot \nabla g_\eta(x-y)\alpha(\tau,y)\dd x\dd y\dd \tau \\
&= 2\int_s^t \iint (\rho v + \omega V)(\tau,x)\cdot \nabla g_\eta(x-y)\alpha(\tau,y)\dd x\dd y\dd \tau.
\end{align*}
Remark that for almost any $\tau \in [0,T]$, we have
\begin{align*}
\iint& |\nabla g_\eta(x-y)\cdot(\rho v + \omega V)(\tau,x)\alpha(\tau,y)|\dd x\dd y \\
&\leq \iint\frac{1}{|x-y|}|(\rho v + \omega V)(\tau,x)|\alpha(\tau,y)| \dd x \dd y \\
&\leq \norm{\rho v + \omega V}_{L^\infty_T L^1}\underset{\tau \in [0,T]}{\sup}\underset{x \in \mathbb{R}^2}{\sup}\int \frac{|\alpha(\tau,y)|}{|x-y|}\dd y \\
&\leq (\norm{\rho}_{L^\infty_T L^1}\norm{v}_{L^\infty_T L^\infty} + \norm{\omega}_{L^\infty_T L^1}\norm{V}_{L^\infty_T L^\infty})R_T[\alpha]\norm{\alpha}_{L^\infty_T L^\infty}.
\end{align*}
where the last inequality follows from the proof of Claim (3) of Proposition \ref{proposition_biot_savart}. Thus by dominated convergence theorem,
\begin{align*}
\iint &g_\eta(x-y)(\alpha(t,x)\alpha(t,y) - \alpha(s,x)\alpha(s,y))\dd x\dd y \\
&\Tend{\eta}{0} 2\int_s^t \iint (\rho v + \omega V)(\tau,x)\cdot \nabla g(x-y)\alpha(\tau,y)\dd x\dd y\dd \tau.
\end{align*}
Combining the upper limit with \eqref{limite_T3_1} we get that
\begin{equation*}
T_3(t) -T_3(s) = 2\int_s^t \iint (\rho v + \omega V)(\tau,x)\cdot \nabla g(x-y)\alpha(\tau,y)\dd x\dd y\dd \tau.
\end{equation*}
Remark that since $\nabla g \ast \alpha = - V^\bot$, we have
\begin{equation*}
\int_s^t\iint (\omega V)(\tau,x)\cdot \nabla g(x-y)\alpha(\tau,y)\dd x\dd y\dd \tau = 0.
\end{equation*}
Finally, we get
\begin{equation*}
T_3(t) -T_3(s) = 2\int_s^t\iint v(\tau,x)\cdot \nabla g(x-y) \rho(\tau,x)\alpha(\tau,y)\dd x\dd y
\end{equation*}
which ends the proof of \ref{derivee_H_termes_3} for almost every $t \in [0,T]$.
\end{proof}

Now for the fourth term, we have:
\begin{claim}\label{derivee_H_termes_4}
\begin{align*}
\frac{\dd T_4}{\dd t}
=& 2\iint V(t,x)\cdot\nabla g(x-y)(\omega_N - \omega)(t,x)\dd x\dd \alpha_N(t,y) \\
&- 2\iint v(t,x)\cdot\nabla g(x-y)\rho(t,x)\dd x \dd \alpha_N(t,y) \\
&-\frac{2}{N}\sum_{i=1}^N p_i \cdot \nabla g \ast \alpha(q_i) \\
=:& \; T_{4,1} + T_{4,2} + T_{4,3}.
\end{align*}
\end{claim}
\begin{proof}[Proof of Claim \ref{derivee_H_termes_4}]
Recall that
\begin{align*}
T_4 =& - 2 \iint_{\mathbb{R}^2\times\mathbb{R}^2}g(x-y)\alpha(t,x)\dd x \dd \alpha_N(t,y) \\
=& - 2 \iint_{\mathbb{R}^2\times\mathbb{R}^2}g(x-y)\alpha(t,x)\omega_N(t,y)\dd \dd y \\
&- \frac{2}{N}\sum_{i=1}^N \int_{\mathbb{R}} g(x-q_i(t))\alpha(t,x)\dd x.
\end{align*}
Using $g_\eta$ in the same way we did for the previous claim, one can prove that $T_4$ is $W^{1,\infty}$ and that for almost every $t \in (0,T)$, 
\begin{align*}
\frac{\dd T_4}{\dd t} 
=& \bigg(-2\iint \nabla g(x-y)\cdot(V(t,x)\omega(t,x) + \rho(t,x)v(t,x))\dd x\dd \alpha_N(t,y) \\
&+ 2\iint \nabla g(x-y) \cdot V_N(t,y)\alpha(t,x)\omega_N(t,y)\dd x\dd y\bigg) \\
&+ \frac{2}{N}\sum_{i=1}^N p_i \cdot \int \nabla g(x - q_i)\alpha(t,x)\dd x \\
=& A_1 + A_2.
\end{align*}

Let us compute each term. For the second term in $A_1$, remark that:
\begin{align*}
&\iint\nabla g(x-y) \cdot V_N(t,y)\alpha(t,x)\omega_N(t,y)\dd x\dd y \\
&= - \iiint \nabla g(x-y)\cdot \nabla^\bot g(y-z)\alpha(t,x)\omega_N(t,y)\dd x\dd y\dd \alpha_N(t,z) \\
&= \iiint \nabla^\bot g(x-y)\cdot \nabla g(y-z)\alpha(t,x)\omega_N(t,y)\dd x\dd y\dd \alpha_N(t,z) \\
&= \iint \left(- \int \nabla^\bot g(y-x)\alpha(t,x)\dd x\right)\cdot \nabla g(y-z)\omega_N(t,y)\dd y\dd \alpha_N(t,z) \\
&= \iint V(t,y)\cdot\nabla g(y-z)\omega_N(t,y)\dd y \dd\alpha_N(t,z) \\
&= \iint V(t,x)\cdot \nabla g(x-y)\omega_N(t,x)\dd x\dd\alpha_N(t,y).
\end{align*}
It follows that
\begin{align*}
A_1
=& 2\iint V(t,x)\cdot\nabla g(x-y)(\omega_N - \omega)(t,x)\dd x\dd \alpha_N(t,y) \\
&- 2\iint v(t,x)\cdot\nabla g(x-y)\rho(t,x)\dd x\dd \alpha_N(t,y) \\
=& T_{4,1} + T_{4,2}.
\end{align*}
For the second term:
\begin{align*}
A_2
&= -\frac{2}{N}\sum_{i=1}^N p_i \cdot \int \nabla g(q_i - x)\alpha(t,x)\dd x \\
&= -\frac{2}{N}\sum_{i=1}^N p_i \cdot \nabla g \ast\alpha(q_i) \\
&= T_{4,3}.
\end{align*}
\end{proof}

We now need to differentiate the fifth term with respect to time.
\begin{claim}\label{derivee_H_termes_5}
$T_5$ is Lipschitz and for almost every $t \in [0,T]$ we have:
\begin{align*}
\frac{\dd T_5}{\dd t}
&= -2\int \nabla \omega \cdot (V-V_N)(\omega -\omega_N) \\
&=: T_{5,1}.
\end{align*}
\end{claim}
\begin{proof}[Proof of Claim \ref{derivee_H_termes_5}]
Let $(\chi_\eta)_{\eta > 0}$ be a sequence of mollifiers with compact support and set $\omega_N^\eta = \chi_\eta\ast\omega_N$. For $t,s \in [0,T]$, we have
\begin{align*}
T_5(t) - T_5(s) 
&= \int |\omega(t) - \omega_N(t)|^2 - \int |\omega(s) - \omega_N(s)|^2 \\
&= \underset{\eta \rightarrow 0}{\text{lim}} \int |\omega(t) - \omega_N^\eta(t)|^2 - \int |\omega(s) - \omega_N^\eta(s)|^2 \\
&= \underset{\eta \rightarrow 0}{\text{lim}} \int_s^\tau \frac{\dd}{\dd\tau}\left( \int |\omega(\tau) - \omega_N^\eta(\tau)|^2\right)\dd\tau
\end{align*}
Now,
\begin{align*}
\frac{\dd}{\dd\tau}\int |\omega(\tau) - \omega_N^\eta(\tau)|^2
=& 2\int(\omega-\omega_N^\eta)\div(\omega V - \chi_\eta \ast(\omega_N V_N)) \\
=& 2\int(\omega-\omega_N^\eta)\div(\omega(V-V_N) + (\omega-\omega_N^\eta)V_N) \\
&+ 2\int(\omega-\omega_N^\eta)\div(\omega_N^\eta V_N - \chi_\eta\ast(\omega_NV_N)) \\
=& 2 \int (\omega-\omega_N^\eta)\nabla\omega\cdot(V-V_N) \\
&+ 2 \int (\omega - \omega_N^\eta)\nabla(\omega-\omega_N^\eta)\cdot V_N \\
&+ 2 \int \omega \div(\omega_N^\eta V_N - \chi_\eta\ast(\omega_NV_N)) \\
&- 2 \int \omega_N^\eta \div(\omega_N^\eta V_N - \chi_\eta\ast(\omega_NV_N)).
\end{align*}
For the first term, remark that for any $1 < p < 2$, we have $V_N \in L^p_{\loc}$ and $\omega_N^\eta \Tend{\eta}{0} \omega_N$ in $L^q$ where $\displaystyle{\frac{1}{p} + \frac{1}{q} = 1}$. Thus since $\omega - \omega_N$ has compact support and $\nabla \omega \in L^\infty$, we have
\begin{align*}
2 \int (\omega-\omega_N^\eta)\nabla\omega\cdot(V-V_N) \Tend{\eta}{0} 2 \int (\omega-\omega_N)\nabla\omega\cdot(V-V_N).
\end{align*}
Remark also that the second term cancels out because $V_N$ is divergent-free:
\begin{equation*}
\int (\omega-\omega_N^\eta)V_N\cdot\nabla(\omega-\omega_N^\eta) = -\frac{1}{2}\int V_N \cdot \nabla |\omega-\omega^\eta_N|^2 \\
= 0
\end{equation*}
We will now prove that the two last term tends to zero. For the third term, we integrate by parts to get:
\begin{equation*}
2 \int \omega \div(\omega_N^\eta V_N - \chi_\eta\ast(\omega_NV_N))
= - 2 \int \nabla\omega \cdot(\omega_N^\eta V_N - \chi_\eta\ast(\omega_NV_N))
\end{equation*}
Since $\omega_N^\eta V_N \Tend{\eta}{0} \omega_N V_N$, $\chi_\eta\ast(\omega_NV_N) \Tend{\eta}{0} \omega_NV_N$ in $L^1$ and $\nabla \omega \in L^\infty$, we have
\begin{equation*}
2 \int \nabla\omega \cdot(\omega_N^\eta V_N - \chi_\eta\ast(\omega_NV_N)) \Tend{\eta}{0} 0.
\end{equation*}
For the last term, since all the $q_i$ are outside of the support of $\omega_N$ (see Remark \ref{solution_edo_bien_definie}), they are also outside of the support of $\omega_N^\eta$ if $\eta$ is small enough. Thus we have:
\begin{equation*}
V_N \in\ W^{1,p}(\supp{\omega_N^\eta})
\end{equation*}
for any $2 < p < +\infty$. By the commutator estimate of DiPerna and Lions in \cite{DiPernaLions} (see \cite[Lemma 2.2]{DeLellis} for more details) we get
\begin{equation*}
[V_N,\chi_\eta\ast]\omega_N \Tend{\eta}{0} 0 \qquad \text{in}  \; L^1_{\loc}.
\end{equation*}
Since $\omega_N^\eta$ is uniformly bounded in $L^\infty$, we obtain
\begin{equation*}
\int\omega_N^\eta[V_N,\chi_\eta\ast]\omega_N \Tend{\eta}{0} 0.
\end{equation*}
which ends the proof of Claim \ref{derivee_H_termes_5}.
\end{proof}

For the sixth term:
\begin{claim}\label{derivee_H_termes_6}
$T_6$ Lipschitz and for almost every $t \in [0,T]$ we have
\begin{align*}
\frac{\dd T_6}{\dd t}
=& 2\iint_{\mathbb{R}^2\times \mathbb{R}^2} \nabla g(x-y) \cdot V(t,x)(\omega-\omega_N)(t,x)(\omega-\omega_N)(t,y)\dd x \dd y \\
&+ 2\int_{\mathbb{R}^2\times \mathbb{R}^2} \nabla g(x-y)\cdot \nabla^\bot g\ast(\omega-\omega_N)(t,x)\omega_N(t,x)\dd x\dd(\rho-\rho_N)(y).
\end{align*}
\end{claim}

\begin{proof}[Proof of Claim \ref{derivee_H_termes_6}]
Using $g_\eta$ in the same way we did for Claim \ref{derivee_H_termes_3}, one can prove that $T_6$ is $W^{1,\infty}$ and that for almost every $t \in (0,T)$, 
\begin{align*}
\frac{\dd T_6}{\dd t} =& 2\iint_{\mathbb{R}^2\times \mathbb{R}^2} \nabla g(x-y) \cdot (V\omega-V_N\omega_N)(t,x)(\omega-\omega_N)(t,y)\dd x \dd y \\
=& 2\iint_{\mathbb{R}^2\times \mathbb{R}^2} \nabla g(x-y) \cdot V(t,x)(\omega-\omega_N)(t,x)(\omega-\omega_N)(t,y)\dd x \dd y \\
&+ 2\iint_{\mathbb{R}^2\times \mathbb{R}^2}\nabla g(x-y) \cdot (V - V_N)(t,x)\omega_N(t,x)(\omega-\omega_N)(t,y)\dd x \dd y.
\end{align*}
Since $V-V_N = -\nabla^\bot g*(\omega-\omega_N + \rho-\rho_N)$ we get
\begin{multline*}
\iint_{\mathbb{R}^2\times \mathbb{R}^2}\nabla g(x-y) \cdot (V - V_N)(t,x)\omega_N(t,x)(\omega-\omega_N)(t,y)\dd x \dd y \\
= \iint_{\mathbb{R}^2\times \mathbb{R}^2} \nabla g(x-y)\cdot \nabla^\bot g\ast(\omega-\omega_N)(t,x)\omega_N(t,x)\dd x\dd(\rho-\rho_N)(y)
\end{multline*}
which ends the proof of Claim \ref{derivee_H_termes_6}.
\end{proof}

Remark that all terms depending on $p_i$ (coming from the equations of Claim \ref{derivee_H_termes_1}, Claim \ref{derivee_H_termes_4} and Equation \eqref{derivee_H_termes_2}) cancels out, that is 
\begin{equation*}
T_{1,2} + T_{2,1} + T_{4,3} = 0.
\end{equation*}

Now let us group all terms of the form 
\begin{equation*}
\iint_{\mathbb{R}^2\times\mathbb{R}^2\backslash \Delta}v(t,x)\cdot \nabla g(x-y) \dd\mu(x)\dd\nu(y)
\end{equation*}
coming from the equations of Claims \ref{derivee_H_termes_1},  \ref{derivee_H_termes_3} and \ref{derivee_H_termes_4}:
\begin{align*}
&T_{1,3} + T_{3,1} + + T_{4,2} \\ 
=& 2\iint_{\mathbb{R}^2\times\mathbb{R}^2\backslash \Delta}v(t,x)\cdot \nabla g(x-y)(\rho_N\otimes(\alpha_N - \alpha) -\rho\otimes\alpha_N + \rho\otimes\alpha)(\dd x\dd y) \\
=& 2\iint_{\mathbb{R}^2\times\mathbb{R}^2\backslash \Delta}v(t,x)\cdot \nabla g(x-y)\dd(\rho-\rho_N)(t,x)\dd(\alpha-\alpha_N)(t,y) \\
=& 2\iint_{\mathbb{R}^2\times\mathbb{R}^2\backslash \Delta}v(t,x)\cdot \nabla g(x-y)\dd(\alpha-\alpha_N)(t,x)\dd(\alpha-\alpha_N)(t,y) \\
&- 2\iint_{\mathbb{R}^2\times\mathbb{R}^2\backslash \Delta}v(t,x)\cdot \nabla g(x-y)(\omega-\omega_N)(t,x)\dd x\dd(\alpha-\alpha_N)(t,y) \\
=& \iint_{\mathbb{R}^2\times\mathbb{R}^2\backslash \Delta}(v(t,x)-v(t,y))\cdot \nabla g(x-y)\dd(\alpha-\alpha_N)(t,x)\dd(\alpha-\alpha_N)(t,y) \\
&+ 2\int v^\bot(t,x)\cdot(V-V_N)(t,x)(\omega-\omega_N)(t,x)\dd x
\end{align*}
because
\begin{align*}
2&\iint_{\mathbb{R}^2\times\mathbb{R}^2\backslash \Delta}v(t,x)\cdot \nabla g(x-y)\dd(\alpha-\alpha_N)(t,x)\dd(\alpha-\alpha_N)(t,y)  \\
&= \iint_{\mathbb{R}^2\times\mathbb{R}^2\backslash \Delta}v(t,x)\cdot \nabla g(x-y)\dd(\alpha-\alpha_N)(t,x)\dd(\alpha-\alpha_N)(t,y) \\
&+ \iint_{\mathbb{R}^2\times\mathbb{R}^2\backslash \Delta}v(t,y)\cdot \nabla g(y-x)\dd(\alpha-\alpha_N)(t,y)\dd(\alpha-\alpha_N)(t,x) \\
&= \iint_{\mathbb{R}^2\times\mathbb{R}^2\backslash \Delta}(v(t,x)-v(t,y))\cdot \nabla g(x-y)\dd(\alpha-\alpha_N)(t,x)\dd(\alpha-\alpha_N)(t,y).
\end{align*}

Let us do the same for $V$ (there is only one term, coming from the equations of Claim \ref{derivee_H_termes_4}):
\begin{align*}
T_{4,1} &=  2\iint V(t,x)\cdot\nabla g(x-y)(\omega_N - \omega)(t,x)\dd x\dd \alpha_N(t,y) \\
&= 2\iint V(t,x)\cdot\nabla g(x-y)(\omega_N - \omega)(t,x)\dd x\dd(\alpha-\alpha_N)(t,y)
\end{align*}
because 
\begin{align*}
&\iint V(t,x)\cdot\nabla g(x-y)(\omega_N - \omega)(t,x)\alpha(t,y)\dd x\dd y \\
&= \int V(t,x)(\omega_N - \omega)(t,x)\cdot V^\bot(t,x)\dd x \\
&= 0.
\end{align*}

Thus
\begin{equation*}
T_{4,1} = -2\iint V^\bot(t,x)\cdot(V - V_N)(t,x)(\omega(t,x)-\omega_N(t,x))\dd x.
\end{equation*}

Putting all terms together, we obtain
\begin{equation}\label{derivee_energie_modulee}
\begin{aligned}
\frac{\dd \mathcal{H}_N}{\dd t} 
=& -\frac{2}{N}\sum_{i=1}^N\nabla v(q_i):(v(q_i)-p_i)^{\otimes 2} \\
&+ \iint_{\mathbb{R}^2\times\mathbb{R}^2\backslash \Delta}(v(t,x)-v(t,y))\cdot \nabla g(x-y)(\alpha-\alpha_N)^{\otimes 2}(\dd x\dd y) \\
&+ \int_{\mathbb{R}^2} A\cdot(V-V_N)(\omega-\omega_N) \\
&+ 2\iint_{\mathbb{R}^2\times \mathbb{R}^2} \nabla g(x-y) \cdot V(t,x)(\omega-\omega_N)(t,x)(\omega-\omega_N)(t,y)\dd x \dd y \\
&+ 2\iint_{\mathbb{R}^2\times \mathbb{R}^2} \nabla g(x-y)\cdot \nabla^\bot g\ast(\omega-\omega_N)(t,x)\omega_N(t,x)\dd x\dd(\rho-\rho_N)(y) \\
=:& \; R_1 + R_2 + R_3 + R_4 + R_5
\end{aligned}
\end{equation}
with $A = 2(v^\bot - V^\bot - \nabla \omega)$. In order to control $R_3$, we will need the following result:
\begin{lemma}\label{lemme_controle_termes_croises}
If $A \in W^{1,\infty}$, then there exists $\lambda > 0$ and a constant $C$ depending only on $\norm{A}_{W^{1,\infty}}$ such that
\begin{multline*}
\bigg|\int A\cdot(V - V_N)(\omega-\omega_N) \bigg| \leq C\bigg(\mathcal{F}(Q_N,\rho) + \norm{\omega-\omega_N}^2_{L^2} \\ 
+ \iint g(x-y)(\omega-\omega_N)(x)(\omega-\omega_N)(y)\dd x \dd y + N^{-\lambda}\bigg)
\end{multline*}
where $\mathcal{F}$ is the functionnal defined by \eqref{definition_energie_modulee_F}.
\end{lemma}

\begin{proof}
Let us fix $I = \int A\cdot(V - V_N)(\omega-\omega_N)$, then
\begin{align*}
I =& -\iint A(x)\cdot\nabla^\bot g(x-y)(\omega-\omega_N)(x)\dd (\omega + \rho - \omega_N - \rho_N)(y)\dd x \\
=& \frac{1}{2}\iint (A^\bot(x)-A^\bot(y))\cdot \nabla g(x-y)(\omega-\omega_N)(x)(\omega-\omega_N)(y)\dd x \dd y \\
&-\int \nabla g \ast[\cdot A^\bot(\omega-\omega_N)](y)\dd(\rho - \rho_N)(y) \\
=:& I_1 + I_2.
\end{align*}
By \cite[Lemma 4.3]{Serfaty},
\begin{align*}
I_1 = c\int \nabla A^\bot :[g\ast(\omega-\omega_N),g\ast(\omega-\omega_N)]
\end{align*}
where for $i,j \in \{1,2\}$ and $h$ regular enough,
\begin{align*}
[h,h]_{i,j} = 2\partial_i h \partial_j h - |\nabla h|^2\delta_{i,j}. 
\end{align*}
Hence
\begin{equation*}
|I_1| \leq C\norm{\nabla A}_{L^\infty}\norm{\nabla g\ast(\omega-\omega_N)}^2_{L^2}.
\end{equation*}
Therefore by \eqref{equality_kinetic_energy_mod_energy} we have
\begin{equation}\label{bound_I_1}
|I_1| \leq C\norm{\nabla A}_{L^\infty}\iint g(x-y)(\omega-\omega_N)(x)(\omega-\omega_N)(y) \dd x \dd y.
\end{equation}
Now denote 
\begin{equation*}
\xi_N := -\nabla g \ast[\cdot A^\bot(\omega-\omega_N)].
\end{equation*}
We can write
\begin{equation*}
I_2 = \int \xi_N(y)(\rho - \rho_N)(\dd y).
\end{equation*}

Using Proposition \ref{coercivite_F} (proved in \cite{Serfaty}), we get that for any $0 <\theta < 1$, there exists constants $C,\lambda > 0$ such that
\begin{align*}
|I_2| \leq& C|\xi_N|_{C^{0,\theta}}N^{-\lambda} \\ 
&+ C\norm{\nabla \xi_N}_{L^2}\bigg(\mathcal{F}(Q_N,\rho) + (1+\norm{\rho}_{L^\infty})N^{-1} + \frac{\ln(N)}{N}\bigg)^\frac{1}{2}.
\end{align*}

By Morrey's inequality (see for example \cite[Theorem 9.12]{Brezis}) and Proposition \ref{Calderon_Zygmund}, for some $p > 2$ depending only on $\theta$, we have
\begin{align*}
|\xi_N|_{C^{0,\theta}} &\leq C\norm{\nabla^2 g \ast[\cdot A^\bot(\omega-\omega_N)]}_{L^p} \\
&\leq C\norm{A(\omega-\omega_N)}_{L^p}  \\
&\leq C\norm{A}_{L^\infty}\norm{\omega-\omega_N}_{L^p} \\
&\leq C\norm{A}_{L^\infty}\left(\norm{\omega^0}_{L^p} + \norm{\omega_N^0}_{L^p}\right)
\end{align*}
by Remark \ref{solution_edo_bien_definie}. Therefore by Assumption \eqref{controle_uniforme_omega_N_0},
\begin{align*}
|\xi_N|_{C^{0,\theta}} &\leq C.
\end{align*}
where $C$ is independent of $N$. Now, by Proposition \ref{Calderon_Zygmund},
\begin{equation*}
\norm{\nabla^2 g\ast[\cdot A^\bot(\omega-\omega_N)]}_{L^2} \leq \norm{A}_{L^\infty}\norm{\omega-\omega_N}_{L^2}.
\end{equation*}
Thus we obtain the inequality we wanted to prove.
\end{proof}

Let us get back to the expression of $\displaystyle{\frac{\dd\mathcal{H}_N}{\dd t} = R_1 + R_2 + R_3 + R_4 + R_5}$ given by \eqref{derivee_energie_modulee}. We have
\begin{equation}\label{borne_R_1_derivee_H_N}
|R_1| \leq \frac{2\norm{\nabla v}_{L^\infty}}{N}\sum_{i=1}^N |v(q_i) - p_i|^2 \leq 2\norm{\nabla v}_{L^\infty}\mathcal{H}_N.
\end{equation}
For the second term,
\begin{align*}
R_2 =& \iint_{\mathbb{R}^2\times\mathbb{R}^2\backslash \Delta}(v(t,x)-v(t,y))\cdot \nabla g(x-y)(\omega-\omega_N)^{\otimes 2}(\dd x\dd y) \\
&+ \iint_{\mathbb{R}^2\times\mathbb{R}^2\backslash \Delta}(v(t,x)-v(t,y))\cdot \nabla g(x-y)(\rho-\rho_N)^{\otimes 2}(\dd x\dd y) \\
&+ 2 \iint_{\mathbb{R}^2\times\mathbb{R}^2\backslash \Delta} (v(t,y)-v(t,x))\cdot \nabla g(x-y) (\omega-\omega_N)(x) \dd x \dd (\rho-\rho_N)(y) \\
=:& R_{2,1} + R_{2,2} + R_{2,3}.
\end{align*}
We can bound  $R_{2,1}$ as we did to obtain Inequality \eqref{bound_I_1} and we get
\begin{equation}\label{borne_R_2_1_derivee_H_N}
|R_{2,1}| \leq C\norm{\nabla v}_{L^\infty}\iint_{\mathbb{R}^2\times\mathbb{R}^2} g(x-y)(\omega-\omega_N)(x)(\omega-\omega_N)(y)\dd x \dd y.
\end{equation}

Using Proposition \ref{main_functional_inequality_serfaty} (proved in \cite{Serfaty}) with $\mu = \rho \in L^\infty$, we get
\begin{equation}\label{borne_R_2_2_derivee_H_N}
R_{2,2} 
\leq C\norm{v}_{W^{1,\infty}} \left(\mathcal{F}(Q_N,\rho) + (1+\norm{\rho}_{L^\infty})N^{-\lambda}\right).
\end{equation}
Now,
\begin{align*}
R_{2,3} = \int \chi_N \dd(\rho-\rho_N) 
\end{align*}
with $\chi_N = -2 v\cdot \nabla g \ast(\omega-\omega_N) + 2\nabla g\ast(\cdot v(\omega-\omega_N))$. Using Proposition \ref{coercivite_F}, we get that
\begin{align*}
|R_{2,3}| \leq&  C\bigg(|\chi_N|_{\mathcal{C}^{0,\theta}}N^{-\lambda} + \norm{\nabla \chi_N}_{L^2}\bigg(\mathcal{F}(Q_N,\rho) \\&+ (1+\norm{\rho}_{L^\infty})N^{-1}+\frac{\ln(N)}{N}\bigg)^\frac{1}{2}\bigg).
\end{align*}
Now by Morrey's inequality (see for example \cite[Theorem 9.12]{Brezis}), Hardy-Littlewood-Sobolev inequality (see for example \cite[Theorem 1.7]{BahouriCheminDanchin}) and Proposition \ref{Calderon_Zygmund}, for some $p > 2$ we have
\begin{align*}
|\chi_N|_{\mathcal{C}^{0,\theta}} 
\leq& C(\norm{\nabla v}_{L^\infty}\norm{\nabla g \ast(\omega-\omega_N)}_{L^p} +\norm{v}_{L^\infty}\norm{\nabla^2 g \ast(\omega-\omega_N)}_{L^p} \\
&+ \norm{\nabla^2 g \ast(\cdot v(\omega-\omega_N))}_{L^p}) \\
&\leq C\norm{v}_{W^{1,\infty}}(\norm{\omega-\omega_N}_{L^p} + \norm{\omega-\omega_N}_{L^\frac{2p}{p+2}}) \\
&\leq C\norm{v}_{W^{1,\infty}}(\norm{\omega-\omega_N}_{L^1} + \norm{\omega-\omega_N}_{L^\infty}) \\
&\leq C\norm{v}_{W^{1,\infty}}(\norm{\omega^0}_{L^1}+\norm{\omega_N^0}_{L^1} + \norm{\omega^0}_{L^\infty} + \norm{\omega_N^0}_{L^\infty})
\end{align*}
by Remark \ref{solution_edo_bien_definie}. Therefore by Assumption \eqref{controle_uniforme_omega_N_0},
\begin{align*}
|\chi_N|_{C^{0,\theta}} &\leq C.
\end{align*}
Moreover, using Proposition \ref{Calderon_Zygmund} and Equation \eqref{equality_kinetic_energy_mod_energy}, we have
\begin{align*}
\norm{\nabla \chi_N}_{L^2} 
\leq& \norm{\nabla v}_{L^\infty}\norm{\nabla g \ast(\omega-\omega_N)}_{L^2} +\norm{v}_{L^\infty}\norm{\nabla^2 g \ast(\omega-\omega_N)}_{L^2} \\
&+ \norm{\nabla^2 g \ast(\cdot v(\omega-\omega_N))}_{L^2} \\
&\leq C\norm{v}_{W^{1,\infty}}\bigg(\bigg(\iint_{\mathbb{R}^2\times\mathbb{R}^2}g(x-y)(\omega-\omega_N)(x)(\omega-\omega_N)(y)\dd x \dd y\bigg)^\frac{1}{2} \\
&+ \norm{\omega-\omega_N}_{L^2}\bigg).
\end{align*}
Therefore
\begin{equation}\label{borne_R_2_3_derivee_H_N}
\begin{aligned}
|R_{2,3}| \leq& C\norm{v}_{W^{1,\infty}}\bigg(\norm{\omega-\omega_N}^2_{L^2} \\ &+\iint_{\mathbb{R}^2\times\mathbb{R}^2}g(x-y)(\omega-\omega_N)(x)(\omega-\omega_N)(y)\dd x \dd y \\
&+ \mathcal{F}(Q_N,\rho) + (1+\norm{\rho}_{L^\infty})N^{-\lambda} \bigg). 
\end{aligned}
\end{equation}
Combining inequalities \eqref{borne_R_2_1_derivee_H_N}, \eqref{borne_R_2_2_derivee_H_N} and \eqref{borne_R_2_3_derivee_H_N} we find that
\begin{equation}\label{borne_R_2_derivee_H_N}
|R_2| \leq C\norm{v}_{W^{1,\infty}}(\mathcal{H}_N + \mathcal{F}(Q_N,\rho) + (1+\norm{\rho}_{L^\infty})N^{-\lambda}).
\end{equation}

Now using Lemma \ref{lemme_controle_termes_croises}, since $V$, $v$ and $\nabla \omega$ are in $L^\infty$,
\begin{equation}\label{borne_R_3_derivee_H_N}
\begin{aligned}
|R_3| \leq& C\bigg(\mathcal{F}(Q_N,\rho) + \norm{\omega-\omega_N}^2_{L^2} + N^{-\lambda} \\
&+ \iint g(x-y)(\omega-\omega_N)(x)(\omega-\omega_N)(y)\dd x \dd y \bigg).
\end{aligned}
\end{equation}
We can bound $R_4$ as we did to obtain Inequality \eqref{bound_I_1} (with $A = V$) and we get
\begin{equation}\label{borne_R_4_derivee_H_N}
|R_4| \leq C\norm{\nabla V}_{L^\infty}\iint g(x-y)(\omega-\omega_N)(x)(\omega-\omega_N)(y) \dd x \dd y.
\end{equation}
Finally we have
\begin{equation*}
R_5 = \int \nabla g \ast(\cdot u_N)\dd(\rho-\rho_N)
\end{equation*}
with $u_N = -\omega_N\nabla^\bot g\ast(\omega-\omega_N)$. Using Proposition \ref{coercivite_F} we get
\begin{align*}
|R_5| \leq& C\bigg(|\nabla g\ast(\cdot u_N)|_{\mathcal{C}^{0,\theta}}N^{-\lambda} + \norm{\nabla^2 g\ast(\cdot u_N)}_{L^2}\bigg(\mathcal{F}(Q_N,\rho) \\
&+ (1+\norm{\rho}_{L^\infty})N^{-1}+\frac{\ln(N)}{N}\bigg)^\frac{1}{2}\bigg)
\end{align*}
Using Morrey's inequality (see for example \cite[Theorem 9.12]{Brezis}), Proposition \ref{Calderon_Zygmund} and Hardy-Littlewood-Sobolev inequality (see for example \cite[Theorem 1.7]{BahouriCheminDanchin}), we get that for some $p > 2$,
\begin{align*}
|\nabla g\ast(\cdot u_N)|_{\mathcal{C}^{0,\theta}} 
&\leq C\norm{\nabla^2 g\ast(\cdot u_N)}_{L^p} \\
&\leq C\norm{u_N}_{L^p} \\
&\leq C\norm{\omega_N}_{L^\infty}\norm{\nabla g \ast(\omega-\omega_N)}_{L^p} \\
&\leq C\norm{\omega_N}_{L^\infty}\norm{\omega-\omega_N}_{L^\frac{2p}{p+2}} \\
&\leq C\norm{\omega_N}_{L^\infty}(\norm{\omega-\omega_N}_{L^1} +  \norm{\omega-\omega_N}_{L^\infty}) \\
&\leq C\norm{\omega_N^0}_{L^\infty}(\norm{\omega^0}_{L^1} +\norm{\omega_N^0}_{L^1} + \norm{\omega^0}_{L^\infty}+\norm{\omega_N^0}_{L^\infty})
\end{align*}
by Remark \ref{solution_edo_bien_definie}. Therefore by Assumption \eqref{controle_uniforme_omega_N_0},
\begin{equation*}
|\nabla g\ast(\cdot u_N)|_{\mathcal{C}^{0,\theta}} \leq C.
\end{equation*}
Now using Proposition \ref{Calderon_Zygmund} again, we get
\begin{align*}
\norm{\nabla^2 g\ast(\cdot u_N)}_{L^2} &\leq C\norm{u_N}_{L^2} \\
&\leq C\norm{\omega_N}_{L^\infty}\iint_{\mathbb{R}^2\times\mathbb{R}^2} g(x-y)(\omega-\omega_N)(x)(\omega-\omega_N)(y) \dd x \dd y
\end{align*}
by Equation \eqref{equality_kinetic_energy_mod_energy}. Therefore by Assumption \eqref{controle_uniforme_omega_N_0},
\begin{equation}\label{borne_R_5_derivee_H_N}
\begin{aligned}
|R_5| \leq& C\bigg(\mathcal{F}(Q_N,\rho) + (1+\norm{\rho}_{L^\infty})N^{-\lambda} \\
&+ \iint_{\mathbb{R}^2\times\mathbb{R}^2} g(x-y)(\omega-\omega_N)(x)(\omega-\omega_N)(y) \dd x \dd y\bigg)
\end{aligned}
\end{equation}

Combining inequalities \eqref{borne_R_1_derivee_H_N}, \eqref{borne_R_2_derivee_H_N}, \eqref{borne_R_3_derivee_H_N}, \eqref{borne_R_4_derivee_H_N} and \eqref{borne_R_5_derivee_H_N} we get
\begin{align*}
\left|\frac{\dd\mathcal{H}_N}{\dd t}\right| \leq& C\bigg(\mathcal{H}_N + \mathcal{F}(Q_N,\rho) + N^{-\beta} \bigg).
\end{align*}
for some $\beta > 0$. We are only remained to bound $\mathcal{F}(Q_N,\rho)$ by $\mathcal{H}_N$. Let us write
\begin{align*}
\mathcal{F}(Q_N,\rho) =& \iint_{\mathbb{R}^2\times\mathbb{R}^2\backslash \Delta} g(x-y)(\omega+\rho-\omega_N - \rho_N)^{\otimes 2}(\dd x \dd y) \\
&-\iint_{\mathbb{R}^2\times\mathbb{R}^2} g(x-y)(\omega-\omega_N)(x)(\omega-\omega_N)(y)\dd x \dd y \\
&+ 2\iint_{\mathbb{R}^2\times\mathbb{R}^2\backslash \Delta} g(x-y)(\omega - \omega_N)(x) \dd x \dd(\rho-\rho_N)(y) \\
&\leq \mathcal{H}_N + 0 + 2\iint_{\mathbb{R}^2\times\mathbb{R}^2\backslash \Delta} g(x-y)(\omega - \omega_N)(x) \dd x \dd(\rho-\rho_N)(y). 
\end{align*}
To bound the upper integral, we use Proposition \ref{coercivite_F} to get that
\begin{equation}\label{controle_nrj_croises}
\begin{aligned}
2\bigg|\int g\ast(\omega-\omega_N)&\dd(\rho-\rho_N)\bigg| \\
\leq& C\bigg(|g\ast(\omega-\omega_N)|_{\mathcal{C}^{0,\theta}}N^{-\lambda} + 2C\norm{\nabla g\ast(\omega-\omega_N)}_{L^2} \\
&\times \bigg(\mathcal{F}(Q_N,\rho) + (1+\norm{\rho}_{L^\infty})N^{-1}+\frac{\ln(N)}{N}\bigg)^\frac{1}{2}\bigg) \\
\leq& C|g\ast(\omega-\omega_N)|_{\mathcal{C}^{0,\theta}}N^{-\lambda} + C\norm{\nabla g\ast(\omega-\omega_N)}^2_{L^2} \\
&+ \frac{1}{2}\bigg(\mathcal{F}(Q_N,\rho) + (1+\norm{\rho}_{L^\infty})N^{-1}+\frac{\ln(N)}{N}\bigg).
\end{aligned}
\end{equation}
Now by Morrey's inequality (see for example \cite[Theorem 9.12]{Brezis}) and Hardy-Littlewood-Sobolev inequality (see for example \cite[Theorem 1.7]{BahouriCheminDanchin}), for some $p > 2$,
\begin{equation}\label{bound_holder_g_ast_omega_omega_N}
\begin{aligned}
|g\ast(\omega-\omega_N)|_{\mathcal{C}^{0,\theta}} &\leq C\norm{\nabla g\ast(\omega-\omega_N)}_{L^p} \\
&\leq C\norm{\omega-\omega_N}_{L^\frac{2p}{p+2}} \\
&\leq C\norm{\omega^0}_{L^\frac{2p}{p+2}}+\norm{\omega^0_N}_{L^\frac{2p}{p+2}} \\
&\leq C(\norm{\omega^0}_{L^1\cap L^\infty}+\norm{\omega^0_N}_{L^1\cap L^\infty}) \\
&\leq C
\end{aligned}
\end{equation}
by Assumption \eqref{controle_uniforme_omega_N_0}. Using \eqref{equality_kinetic_energy_mod_energy} we also have
\begin{align*}
\norm{\nabla g\ast(\omega-\omega_N)}^2_{L^2} \leq \mathcal{H}_N.
\end{align*}
Therefore 
\begin{align*}
\mathcal{F}(Q_N,\rho) \leq CN^{-\lambda} + C\mathcal{H}_N + \frac{1}{2}\mathcal{F}(Q_N,\rho)
\end{align*}
for some for some $\lambda > 0$, hence 
\begin{equation}\label{bound_F_N_H_N}
\mathcal{F}(Q_N,\rho) \leq C(\mathcal{H}_N + N^{-\lambda}).
\end{equation}
It follows that
\begin{align*}
\frac{\dd\mathcal{H}_N}{\dd t} \leq& C\bigg(\mathcal{H}_N + N^{-\beta} \bigg).
\end{align*}
Applying Grönwall's lemma we get
\begin{equation*}
\mathcal{H}_N(t) \leq (\mathcal{H}_N(0) + CN^{-\beta})e^{CT},
\end{equation*}
that is Inequality \eqref{controle_energie_modulee}.

\subsection{Proof of Proposition \ref{proposition_coerciveness_energy}}

Let $\phi$ be a smooth test function with compact support in $\mathbb{R}^4$. We have:
\begin{align*}
&\int \phi(x,\xi)\left(\frac{1}{N}\sum_{i=1}^N \delta_{(q_i,p_i)} - \rho\otimes\delta_{\xi = v(x)}\right)(\dd x\dd\xi) \\
=& \frac{1}{N}\sum_{i=1}^N[\phi(q_i,p_i) - \phi(q_i,v(q_i))]\\
&+  \frac{1}{N}\sum_{i=1}^N\phi(q_i,v(q_i))- \int \phi(x,v(x))\rho(x)\dd x \\
=:& \; T_1 + T_2.
\end{align*}

Let us bound $T_1$:
\begin{align*}
|T_1| &\leq \frac{1}{N}\sum_{i=1}^N \norm{\phi(q_i,\cdot)}_{W^{1,\infty}}|p_i - v(q_i)| \\
&\leq \norm{\phi}_{H^5}\frac{1}{N}\sum_{i=1}^N |p_i - v(q_i)|
\end{align*}
by Sobolev embedding. Using Cauchy-Schwarz inequality we get
\begin{equation}\label{terme_1_proposition_coercivite}
|T_1| \leq C\norm{\phi}_{W^{1,\infty}}\left(\frac{1}{N}\sum_{i=1}^N|p_i - v(q_i)|^2\right)^\frac{1}{2}.
\end{equation}
For the second term:
\begin{align*}
|T_2| &= \left|\int \phi(x,v(x))(\rho-\rho_N)(\dd x)\right| \\
&\leq \norm{\phi\circ(I_d,v)}_{H^2}\norm{\rho-\rho_N}_{H^{-2}}.
\end{align*}

Let $f := (I_d,v)$, then
\begin{equation}
\label{norme_l2_psi_circ_f}
\norm{\phi \circ f}_{H^2} \leq C(1+\norm{\nabla v}_{W^{1,\infty}}^2)\underset{0 \leq k \leq 2}{\sup} \norm{\nabla^k \phi\circ f}_{L^2}.
\end{equation}

Now let $\psi := \partial^\alpha \phi$ for some multi-index $\alpha$ of length $k \in \{0,1,2\}$, then
\begin{align*}
\norm{\psi\circ f}_{L^2}^2 &= \int |\psi(x,v(t,x))|^2\dd x \\
&\leq \underset{y}{\sup}\int |\psi(x,y)|^2\dd x \\
&\leq \norm{F}_{W^{3,1}}
\end{align*}
where $F(y) := \int |\psi(x,y)|^2\dd x$ (see for example \cite[Corollary 9.13]{Brezis}). Remark that
\begin{align*}
|\partial_{y_i}F(y)| &= 2\left|\int \partial_{y_i}\psi(x,y) \psi(x,y) \dd x\right| \\
&\leq \int  |\partial_{y_i}\psi(x,y)|^2\dd x + \int  |\psi(x,y)|^2\dd x.
\end{align*}

Doing the same computations for all derivatives of $F$ of order less or equal than three and integrating in $y$ gives us
\begin{equation*}
\norm{F}_{W^{3,1}} \leq C\norm{\psi}_{H^3}^2.
\end{equation*}
From \eqref{norme_l2_psi_circ_f} and the upper equation we get
\begin{equation}\label{controle_phi_rond_f}
\norm{\phi \circ f}_{H^2} \leq C(1+\norm{\nabla v}_{W^{1,\infty}}^2)\norm{\phi}_{H^5}.
\end{equation}

Now, by \cite[Proposition 3.10]{Rosenzweig} (which is a refined version of Proposition \ref{coercivite_F}), we have
\begin{align*}
\norm{\rho - \rho_N}_{H^{-2}} &\leq C(|\mathcal{F}(Q_N,\rho)|^\frac{1}{2} + N^{-\frac{1}{2}}|\ln(N)|^\frac{1}{2} + (1+\norm{\rho}_{L^\infty})N^{-\frac{1}{2}}).
\end{align*}
Using Assumption \eqref{controle_uniforme_omega_N_coerc} we can bound $\mathcal{F}(Q_N,\rho)$ as in \eqref{bound_F_N_H_N} to get
\begin{equation}\label{bound_F_H_coerc}
|\mathcal{F}(Q_N,\rho)| \leq C(\mathcal{H}(\omega,\rho,v,\omega_N,Q_N,P_N) + N^{-\lambda})
\end{equation}
and therefore
\begin{align*}
\norm{\rho - \rho_N}_{H^{-2}} &\leq C(\mathcal{H}(\omega,\rho,v,\omega_N,Q_N,P_N)^\frac{1}{2} + (1+\norm{\rho}_{L^\infty})N^{-\lambda})
\end{align*}
for some $\lambda > 0$. Combining the upper inequality with \eqref{terme_1_proposition_coercivite} and \eqref{controle_phi_rond_f} we get that 
\begin{multline*}
\left|\int \phi(x,\xi)\left(\frac{1}{N}\sum_{i=1}^N \delta_{(q_i,p_i)} - \rho\otimes\delta_{\xi=v(x)}\right)\right| \\
\leq  C(1+\norm{\nabla v}_{W^{1,\infty}}^2)\norm{\phi}_{H^5}\big(\mathcal{H}(\omega,\rho,v,\omega_N,Q_N,P_N)^\frac{1}{2}
+ C(1+\norm{\rho}_{L^\infty})N^{-\beta}\big)^\frac{1}{2}
\end{multline*}
for some $\beta > 0$. Thus we get \eqref{coerciveness_energy}. It follows from this estimate that if 
\begin{equation*}
\mathcal{H}(\omega,\rho,v,\omega_N,Q_N,P_N) \underset{N \rightarrow \infty}{\longrightarrow} 0
\end{equation*}
then
\begin{equation*}
\frac{1}{N}\sum_{i=1}^N \delta_{(q_i,p_i)} \underset{H^{-5}}{\longrightarrow} \rho\otimes \delta_{\xi = v(t,x)}.
\end{equation*}
By equality \eqref{equality_kinetic_energy_mod_energy} we also have
\begin{align*}
\norm{\nabla g\ast(\omega_N - \omega)}^2_{L^2} + \norm{\omega - \omega_N}^2_{L^2} \Tend{N}{+\infty} 0
\end{align*}
Now remark that for any $\mu \in L^2$,
\begin{equation}\label{equality_H_minus_1_kinetic_energy}
\begin{aligned}
\norm{\nabla g \ast \mu}^2_{L^2} &= C\norm{\widehat{\nabla g}\widehat{\mu}}^2_{L^2} \\
&= C\int \frac{|\widehat{\mu}(\xi)|^2}{|\xi|^2}\dd \xi \\
&= C\norm{\mu}^2_{\dot H^{-1}}
\end{aligned}
\end{equation}
and therefore
\begin{equation*}
\omega_N - \omega \Tend{N}{+\infty} 0 \; \; \text{in} \; L^2\cap \dot H^{-1}.
\end{equation*}
Finally, using Inequality \eqref{bound_F_H_coerc}, we have
\begin{equation*}
\mathcal{F}(Q_N,\rho) \Tend{N}{+\infty} 0
\end{equation*}
and therefore by \cite[Proposition 3.10]{Rosenzweig} we get that for any $a < -1$,
\begin{equation*}
\rho_N \Tend{N}{+\infty} \rho  \; \; \text{in} \; H ^{a}
\end{equation*}
which concludes the proof of Proposition \ref{proposition_coerciveness_energy}.

\subsection{Proof of Proposition \ref{proposition_donnees_initiales_bien_preparees}}

We have
\begin{align*}
\int_{(\mathbb{R}^2\times\mathbb{R}^2)\backslash \Delta} &g(x-y)(\rho_0 + \omega_0 - \rho_{N,0} - \omega_{N,0})^{\otimes 2}(\dd x\dd y) \\
=& \mathcal{F}(Q_{N,0},\rho_0) + \int_{(\mathbb{R}^2\times\mathbb{R}^2)\backslash \Delta} g(x-y)(\omega_0 - \omega_{N,0})^{\otimes 2}(\dd x\dd y) \\
&-2\int_{(\mathbb{R}^2\times\mathbb{R}^2)\backslash \Delta} g(x-y)(\omega_0 - \omega_{N,0})(x)\dd x\dd(\rho_0-\rho_{N,0})(y)
\end{align*}
It is proved in Theorem 1.1 of \cite{Duerinckx} that the weak-$\ast$ convergence of $\rho_{N,0}$ to $\rho_0$ and the convergence of 
\begin{equation*}
\frac{1}{N^2}\sum_{1 \leq i \neq j \leq N} g(q_i^0 - q_j^0)
\end{equation*}
to 
\begin{equation*}
\iint_{\mathbb{R} ^2\times\mathbb{R}^2} g(x-y)\rho_0(x)\rho_0(y)\dd x \dd y
\end{equation*}
ensures that
\begin{equation}\label{convergence_F_init_0}
\mathcal{F}(Q_{N,0},\rho_0) \Tend{N}{+\infty} 0.
\end{equation}
Using \eqref{equality_kinetic_energy_mod_energy}, \eqref{equality_H_minus_1_kinetic_energy} and the convergence of $\omega_{N,0}$ to $\omega_0$ in $\dot H^{-1}$ we have that
\begin{equation*}
\int_{(\mathbb{R}^2\times\mathbb{R}^2)\backslash \Delta} g(x-y)(\omega_0 - \omega_{N,0})^{\otimes 2}(\dd x\dd y) \Tend{N}{+\infty} 0.
\end{equation*}
Using inequalities \eqref{controle_nrj_croises} and \eqref{bound_holder_g_ast_omega_omega_N} we have 
\begin{align*}
\bigg|&\int_{(\mathbb{R}^2\times\mathbb{R}^2)\backslash \Delta} g(x-y)(\omega_0 - \omega_{N,0})(x)\dd x\dd(\rho_0-\rho_{N,0})(y)\bigg| \\
\leq& C\bigg((\norm{\omega_0}_{L^1\cap L^\infty}+\norm{\omega_{N,0}}_{L^1\cap L^\infty})N^{-\lambda} \\ 
&+\norm{\nabla g\ast(\omega_0-\omega_{N,0})}_{L^2}\bigg(\mathcal{F}(Q_{N,0},\rho_0) + (1+\norm{\rho_0}_{L^\infty})N^{-1}+\frac{\ln(N)}{N}\bigg)^\frac{1}{2}\bigg).
\end{align*}
Using Assumption \eqref{controle_uniforme_omega_N_0_well_prep_init_data}, equations \eqref{equality_kinetic_energy_mod_energy}, \eqref{equality_H_minus_1_kinetic_energy}, \eqref{convergence_F_init_0} and the convergence of $\omega_{N,0}$ to $\omega_0$ in $\dot H^{-1}$ we get that
\begin{equation*}
\int_{(\mathbb{R}^2\times\mathbb{R}^2)\backslash \Delta} g(x-y)(\omega_0 - \omega_{N,0})(x)\dd x \dd(\rho_0-\rho_{N,0})(y) \Tend{N}{+\infty} 0
\end{equation*}
which ends the proof of Proposition \ref{proposition_donnees_initiales_bien_preparees}.


\end{document}